\documentclass[12pt, a4paper]{article}

\usepackage{amsmath}\multlinegap=0pt
\usepackage[latin1]{inputenc}
\usepackage{amsthm}
\usepackage{amssymb}
\usepackage{graphicx}
\usepackage{xcolor}

\numberwithin{equation}{section}
\theoremstyle{plain}
\newtheorem{thm}{Theorem}[section]
\newtheorem{coro}[thm]{Corollary}
\newtheorem{prop}[thm]{Proposition}
\newtheorem{lem}[thm]{Lemma}
\newtheorem{defi}[thm]{Definition}

\theoremstyle{definition}

\theoremstyle{remark}

\newcommand\N{{\mathbb N}}

\newcommand{\R}{\mathbb{R}}
\newcommand{\C}{\mathbb{C}}

\newcommand\pref[1]{(\ref{#1})}

\let \eps\varepsilon

\newcommand{\calP}{{\mathcal P}}
\newcommand{\calN}{{\mathcal N}}
\newcommand{\bfp}{{\mathbf p}}
\newcommand{\ttr}{{\tt r}}
\newcommand{\ttg}{{\tt g}}
\newcommand{\ttb}{{\tt b}}

\newcommand{\parti}{{\mathrm{Part}}}

\DeclareMathOperator{\id}{id}

\def\<#1,#2>{\left<#1,#2\right>}

\newcommand\xx{\mathbf{x}}

\def\PP{{\cal P}}
\def\PPs{{{\cal P}}_{\mathrm{sym}}}
\def\TV{\mathrm{TV}}

\def\PPnr{{{\cal P}}_{N-\mathrm{rep}}}

\newcommand\calS{{\mathcal S}}
\newcommand\calI{{\mathcal I}}
\newcommand\ext{{\mathrm{ext}\,}}

\title{Convex geometry of finite exchangeable laws and de Finetti style representation 
with universal correlated corrections
}

\author{Guillaume Carlier\thanks{CEREMADE, UMR CNRS 7534, Universit\'e Paris IX
Dauphine 
and INRIA-Paris, MOKAPLAN,
\texttt{carlier@ceremade.dauphine.fr}}
\and
Gero Friesecke\thanks{Faculty of Mathematics, Technische Universit\"at M\"unchen,
\texttt{gf@ma.tum.de}}
\and
Daniela V\"{o}gler\thanks{Faculty of Mathematics, Technische Universit\"at M\"unchen,
\texttt{voegler@ma.tum.de}}
}

\begin{document}

\maketitle

\begin{abstract} 
We present a novel analogue for finite exchangeable sequences of the de Finetti, Hewitt and Savage theorem and investigate its implications for multi-marginal optimal transport (MMOT) and Bayesian statistics. If $(Z_1,...,Z_N)$ is a finitely exchangeable sequence of $N$ random variables taking values in some Polish space $X$, we show that the law $\mu_k$ of the first $k$ components has a representation of the form 
$$ 
     \mu_k=\int_{\PP_{\frac{1}{N}}(X)} F_{N,k}(\lambda) \, \mbox{d} \alpha(\lambda)
$$
for some probability measure $\alpha$ on the set of $\frac{1}{N}$-quantized probability measures on $X$ and certain universal polynomials $F_{N,k}$. The latter consist of a leading term $N^{k-1}\! /\prod_{j=1}^{k-1}(N\! -\! j) \lambda^{\otimes k}$ and a finite, exponentially decaying series of correlated corrections of order $N^{-j}$ ($j=1,...,k$). The $F_{N,k}(\lambda)$ are precisely the extremal such laws, expressed via an explicit polynomial formula in terms of their one-point marginals $\lambda$.  
Applications include novel approximations of MMOT via polynomial convexification and the identification of   
the remainder which is estimated in the celebrated error bound of Diaconis-Freedman \cite{Dia-Free} between finite and infinite exchangeable laws.  

\end{abstract}

\medskip


\section{Introduction}\label{sec-intro} 

Multi-marginal optimal transport (MMOT) has attracted a great deal of attention in recent years. The relevance of MMOT to tackle challenging problems arising from electronic density functional theory was established in \cite{CFK, BDG}. In this context, one has to find the joint density of $N$ electrons with fixed one-point marginal so as to minimize a total repulsive Coulombian cost. Even though the problem is difficult for large $N$, it is symmetric (invariant under permutations of the electrons) and only depends on the two-point marginal of the joint law of the $N$ electrons ($2$-body interaction). Whether symmetries and few-body interactions are helpful to analyze such MMOT problems is a natural question. An interesting result from \cite{CFP} relying on the fact that the Coulomb potential has a positive Fourier transform and the de Finetti, Hewitt and Savage theorem is that when one lets $N$ go to $+\infty$, the optimal plan is the independent (infinite product) measure. This is in striking contrast with the more standard two-marginal optimal transport where, for typical costs including the Coulomb cost, optimal plans are sparse and concentrate on low-dimensional subsets of the product space \cite{Brenier, Gangbo-McCann, CFK}. The present paper is motivated by MMOT for a possibly large but finite number of marginals $N$ and symmetric $k$-body (with $k\leq N$) interaction cost. We present a novel explicit analogue of  the de Finetti, Hewitt and Savage theorem and investigate its implications for such problems. We also briefly indicate implications for Bayesian statistics. 

The main technical novelty in our work is the construction of an explicit polynomial inverse of the marginal map from extremal $N$-representable $k$-point probability measures (see below for terminology) to $1$-point probability measures. This extends previous results for $2$-point \cite{GFDV} and $3$-point \cite{KhooYing} measures on finite state spaces to arbitrary $k$ and general Polish spaces. 

Our ensuing finite version of de Finetti 
yields a complete, finite, exponentially decaying series of correlated corrections which need to be added to the independent measure in the case of finite $N$. This explicitly identifies the remainder estimated in the celebrated error bound of Diaconis-Freedman \cite{Dia-Free} between finite and infinite exchangeable laws.

In the remainder of this introduction we first recall the celebrated de Finetti-Hewitt-Savage theorem, then describe in more detail what changes in the finite exchangeable case.

\smallskip

{\bf De Finetti-Hewitt-Savage.} Recall that a sequence $(Z_i)_{i\in\N}$ of random variables taking values in a Polish space $X$ is called exchangeable if the law of $(Z_1,Z_2,...)$ equals that of $(Z_{\sigma(1)},Z_{\sigma(2)},...)$ for each finite permutation $\sigma$ of $\N$, that is each permutation which leaves all but finitely many elements unchanged. The de Finetti-Hewitt-Savage theorem says that any such sequence is a convex mixture of i.i.d. sequences. In other words, the law of $(Z_1,Z_2,...)$ is a convex combination of independent measures, 
\begin{equation} \label{intro:1}
  \mu = \int_{\calP(X)} \lambda^{\otimes\infty} d\alpha(\lambda)
\end{equation}
for some probability measure (or prior) $\alpha$ on the set $\calP(X)$ of probability measures on $X$. In Bayesian language, this says that the general infinite exchangeable sequence $(Z_i)$ is obtained by first picking some distribution $\lambda$ on $X$ at random from some prior, then taking $(Z_i)$ to be i.i.d. with distribution $\lambda$. For comprehensive reviews of exchangeability we refer the reader to Aldous \cite{Aldous85} and Kallenberg \cite{Kal05}. 

{\bf Finite exchangeability; finite extendibility.} A sequence $(Z_1,...,Z_N)$ is called finitely exchangeable if its law equals that of $(Z_{\sigma(1)},...,Z_{\sigma(N)})$ for any permutation $\sigma$ of $\{1,...,N\}$. For $k\le N$, a sequence $(Z_1,...,Z_k)$ is called finitely extendible if its law equals that of the first $k$ elements of some finitely exchangeable sequence $(\tilde{Z}_1,...,\tilde{Z}_N)$.\footnote{Analogously, $(Z_1,...,Z_k)$ is called infinitely extendible if its law equals that of the first $k$ elements of an infinite exchangeable sequence $(\tilde{Z}_1,\tilde{Z}_2,...)$.}

For finite exchangeable sequences $(Z_1,...,Z_N)$ it is well-known that the analogous representation to \eqref{intro:1} with $\lambda^{\otimes\infty}$ replaced by 
$\lambda^{\otimes N}$ does not hold, see Diaconis \cite{Diaconis77},Diaconis and Freedman \cite{Dia-Free},  Jaynes \cite{Ja86}; the error is known to be of order $\frac{1}{N}$ in total variation \cite{Dia-Free,Bobkov}. 

The main approach for describing finite exchangeable sequences which has been introduced in the probability literature is to write such a sequence as a superposition of i.i.d. sequences but drop the requirement that the superposition of the laws be convex, i.e. allows signed measures $\alpha$ in \eqref{intro:1}, see Dellacherie and Meyer \cite{DellaMeyer}, Jaynes \cite{Ja86}, Kerns and Sz\'{e}kely \cite{KS06}, Janson, Konstantopoulos and Yuan \cite{JKT16}. For applications, this approach has limited appeal, for two reasons. First, the signed measure representation is not unique. Second, the superposition does not yield a probability measure for an arbitrary signed $\alpha$, but remaining within probability measures is mandatory for recovering an exchangeable law by sampling (see below)
and for our application to MMOT.

A very interesting second picture of finite exchangeability which appears not to have received the attention it deserves can be found in Kerns and Sz\'{e}kely \cite{KS06} and Kallenberg \cite{Kal05}. Namely, 
finite exchangeable sequences are convex mixtures of ``urn sequences'', or equivalently, finitely exchangeable laws are convex superpositions of symmetrized Dirac measures, the latter being the laws of urn sequences (as described further below).
See Kerns and Sz\'{e}kely (\cite{KS06}, top of p.600), where such a representation appears as an intermediate step in the proof of the signed-measure representation. The laws of urn sequences are known to be the extreme points of the convex set of finitely exchangeable laws, see \cite{KS06} for an elementary proof for finite state spaces and Kallenberg (\cite{Kal05} Proposition 1.8) for a general proof using advanced probability methods. Although not given in \cite{Kal05}, the Kerns-Sz\'{e}kely representation could be deduced from the statement of Proposition 1.8 via disintegration of measures. 

The present work builds upon this picture, which turns out to be very useful for the applications we have in mind. Thus we view finite exchangeable laws as convex mixtures of urns. But we re-instate the idea from original de Finetti, kept in the signed-measure approach, that the parameter space of the superposition should consist of probability measures on the {\it original} Polish space $X$, not its $N$-fold product. In principle, this is possible by parametrizing urn laws by their one-point marginals, which are easily seen to be in $1$-$1$ correspondence with these laws. In practice, to arrive at an explicit representation one needs an explicit formula for the inverse of this marginal map. By deriving such a formula, we obtain a unique representation of finitely exchangeable laws which sheds some light on their universal correlation structure and is useful for applications.

%


{\bf Main results.} In terms of laws, $N$-extendibility turns into what has been called $N$-representability \cite{FMPCK13}: for $k\le N$, a $k$-point probability measure $\mu_k$ on $X^k$, or $k$-plan for short, is called $N$-representable if it is the $k$-point marginal of a symmetric $N$-point probability measure $\mu_N$ on $X^N$ (see Definition \ref{DefNrep}). 

As a first main result, we explicitly determine the extremal $N$-representable $k$-plans, that is, those that cannot be written as strict convex combinations of any other $N$-representable $k$-plans, and give a polynomial parametrization in terms of their one-point marginals. Focusing in this introduction for simplicity on the case $k=4$, these are the probability measures
\begin{align*}
   F_{N,4}(\lambda ) = & \; \frac{N^3}{(N\! -\! 1)(N\! -\! 2)(N\! -\! 3)} \left[ 
     \lambda^{\otimes 4} 
   - \frac{6}{N} S_4 \id^{\otimes 2}_{\#}\!\lambda \, \otimes \lambda^{\otimes 2}\right. \\
   & \left. \hspace*{2.5cm} + \frac{8 \, S_4\id^{\otimes 3}_\# \!\lambda \otimes \, \lambda 
         + 3\, S_4 \id^{\otimes 2}_\# \!\lambda \otimes \, \id^{\otimes 2}_\# \!\lambda}{N^2}
         - \frac{6}{N^3} \id^{\otimes 4}_\# \!\lambda 
    \right]
\end{align*}
where $\lambda$ is a $\tfrac{1}{N}$-quantized probability measure on $X$, i.e. an empirical measure of the form $\tfrac{1}{N}\sum_{i=1}^N\delta_{x_i}$ for some -- not necessarily distinct -- points $x_i\in X$. It is not obvious, but part of our result, that these measures are nonnegative, and different for different $\lambda$. 

The above expression can be viewed as a degree-4 symmetric polynomial in $\lambda$. Besides an overall positive prefactor, the polynomial has leading term $\lambda^{\otimes 4}$ which is homogeneous of degree $4$ and uncorrelated, and alternating corrections of order $\tfrac{1}{N^j}$ which are homogeneous of degree $4\! -\! j$ and more and more strongly correlated. As $N$ tends to infinity $F_{N,4}(\lambda)$ approaches the independent measure $\lambda^{\otimes 4}$, recovering the basis in the de Finetti representation for infinitely representable 4-plans implied by \eqref{intro:1}. The correlated corrections are of significant size even when $N$ is quite large; see Figure \ref{F:poly}. All these findings persist for general $k$; see Theorem \ref{T:expan} 
for the general expression $F_{N,k}(\lambda)$ for extremal $N$-representable $k$-plans. 
Qualitatively, the corrections to independence form a finite exponentially decaying series; quantitatively the (rational) coefficients which appear can be related to the analytic continuation of the Ewens function from genetics, which we introduce for this purpose. 

\begin{figure}[http!]
\begin{center}
    \includegraphics[width=\textwidth]{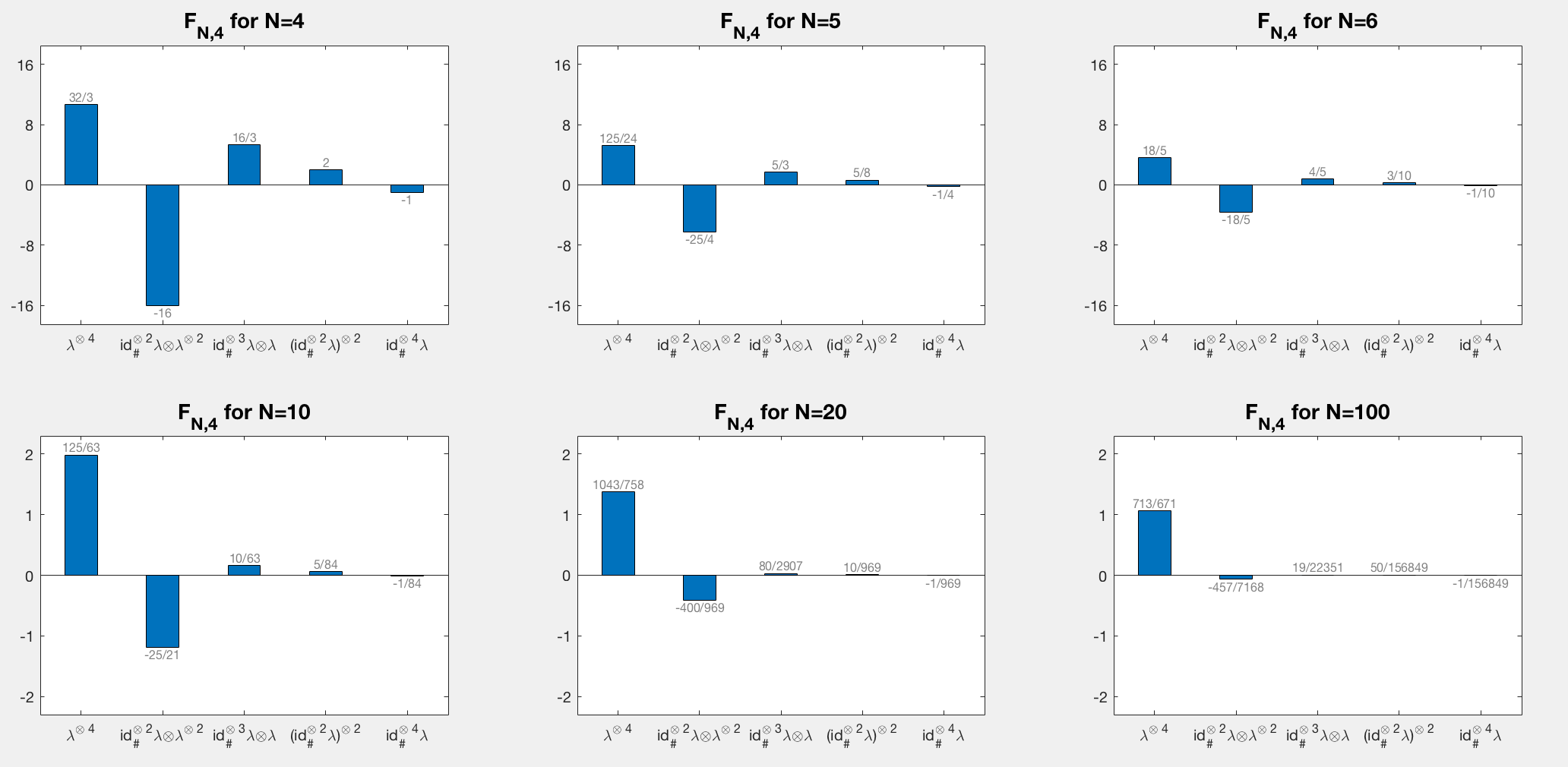}
    \caption{Coefficients of the universal polynomial $F_{N,4}$ for different $N$. For $N=5$ and $6$, the second (correlated) term is bigger respectively equal in absolute value to the first (independent) term; for $N=20$ its size is about 30$\%$ that of the first term. For large $N$, $F_{N,4}$ converges to the independent measure  $\lambda^{\otimes 4}$, but even for $N=100$ the deviation from the latter is still visible.}
\end{center}
\label{F:poly}
\end{figure}

Our second contribution is to cast the abstract insight  \cite{KS06, Kal05} that finite exchangeables are convex mixtures of urn sequences into a quantitative polynomial formula. We show that 
any $N$-representable $k$-plan is a convex mixture of the $F_{N,k}(\lambda)$. More precisely, a $k$-plan is $N$-representable if and only if it is of the form
\begin{equation}\label{intro:2}
  \mu_k=\int_{\PP_{\frac{1}{N}}(X)} F_{N,k}(\lambda) \, \mbox{d} \alpha(\lambda)
\end{equation}
for some probability measure $\alpha$ on the set $\PP_{\frac{1}{N}}(X)$ of $\tfrac{1}{N}$-quantized probability measures on $X$. Moreover if $N=k$ the measure $\alpha$ is unique, giving a one-to-one parametrization of the laws of finitely exchangeable sequences. By contrast, the signed measure representation of such laws is not unique \cite{JKT16}, and not all signed measures give rise to such a law.

Formula \eqref{intro:2} generalizes de Finetti-Hewitt-Savage, \eqref{intro:1}, from infinitely to finitely representable measures, or from infinite to finite exchangeable sequences of random variables.  Formally, in the limit $N\to\infty$ the domain of integration in \eqref{intro:2} tends to all of $\PP(X)$, and the integrand tends to the independent measure $\lambda^{\otimes k}$, recovering de Finetti (see Section 6 for a rigorous account). In  Bayesian  language, formula \eqref{intro:2} says that the general $N$-extendible sequence $(Z_1,...,Z_k)$ of $X$-valued random variables is obtained by first picking some $\tfrac{1}{N}$-quantized distribution $\lambda$ on $X$ at random from some prior, then sampling $(Z_1,...,Z_k)$ from the correlated distribution $F_{N,k}(\lambda)$. In particular, by setting $k\! =\! N$ we conclude  that {\it the general finite exchangeable sequence $(Z_1,...,Z_N)$ is obtained by picking $\lambda$ at random from its -- in the case $k\! =\! N$ unique -- prior $\alpha$ in \eqref{intro:2}, then sampling $(Z_1,...,Z_N)$ from $F_{N,N}(\lambda)$}.   

Sampling from $F_{N,k}(\lambda)$ has a transparent probabilistic meaning which we now explain in the language of urns. Write a given $\tfrac{1}{N}$-quantized probability measure $\lambda$ as $\tfrac{1}{N}\sum_{i=1}^N \delta_{x_i}$ for $N$ not necessarily distinct points $x_1,...,x_N\in X$. Now pick, in turn, $k$ of these points at random {\it without replacement}, and denote the so-obtained sequence by $(Z_1,...,Z_k)$. By construction, the law of this sequence is the $k$-point marginal $\mu_k$ of the symmetrization of the Dirac measure $\delta_{(x_1,...,x_N)}$ on $X^N$. But the polynomial $F_{N,k}$ is precisely constructed as the inverse of the marginal map $\mu_k\mapsto\lambda$ (see eq.~\eqref{defmuk}, eq.~\eqref{lambdaempir}, and Theorem \ref{T:expan} below). Hence $F_{N,k}(\lambda)=\mu_k$, and so $(Z_1,...,Z_k)$ is the sought-after finite $N$-extendible sequence. We find it quite remarkable that the extremal $N$-representable $k$-plans  $F_{N,k}(\lambda)$ -- which emerge purely from convex geometric considerations -- have such a simple  probabilistic meaning, 
being the laws 
of classical examples \cite{Aldous85} of finite exchangeable sequences which are not infinitely extendible.



{\bf Recovering the prior from sampling.}
A nice aspect of our representation of the law of a general finitely exchangeable sequence $(Z_1,...,Z_N)$ (eq.~\eqref{intro:2} with $k=N$) is that the prior $\alpha$, which is unique when $k=N$, can be determined by sampling, as follows. Let 
$$
  \bigl(Z_1^{(\nu)},...,Z_N^{(\nu)}\bigr)_{\nu=1}^n
$$
be a sequence of $n$ independent samples in $X^N$. Form the $\PP_{\frac{1}{N}}(X)$-valued sequence 
$$
  \lambda^{(\nu)} = \frac{1}{N}\sum_{i=1}^N\delta_{Z_i^{(\nu)}}.
$$
Then the empirical measure 
$$
  \frac{1}{n}\sum_{\nu=1}^n \delta_{\lambda^{(\nu)}}
$$
converges almost surely to $\alpha$. See Corollary \ref{C:sampling}.


\smallskip

The paper is organized as follows. After introducing some notations and preliminaries in Section \ref{sec-prel}, Section \ref{sec-quanti} deals with $\frac{1}{N}$-quantized measures. In Section \ref{sec-rep}, we focus on the finite case; we first recall the results of \cite{GFDV} and then identify the universal correlated polynomials $F_{N,k}(\lambda)$. Section \ref{sec-cts} extends these findings to the case of a Polish state space $X$, and show in addition that the $F_{N,k}(\lambda)$ are in fact exposed $N$-representable $k$-plans. Section \ref{sec-hs} discusses connections with the Hewitt and Savage theorem and the  Diaconis-Freedman error bounds.  Section \ref{sec-ewens} gives an unexpected connection with the Ewens sampling formula from genetics. Finally, Section \ref{sec-opti} is devoted to applications to MMOT emphasizing the connection with convexification of polynomials. 

\section{Preliminaries and notations}\label{sec-prel}

In the sequel $X$ will denote a Polish (i.e., complete and separable metric) space. The principal example we have in mind is $X=\R^d$, in which case all of our results are already new and interesting. In this case the metric is the usual Euclidean metric
$d(x,y)=|x-y|=(\sum_{i=1}^d (x_i-y_i)^2))^{1/2}$.  
We denote by $\PP(X)$ the set of Borel probability measures on $X$. Probability measures on $X^k$ will be called $k$-plans. From now on, we fix two integers $k$ and $N$ with $1\le k \le N$. Given $\gamma \in \PP(X^N)$ we denote by $M_k\gamma$ the $k$-point-marginal of $\gamma$, i.e., 
\begin{equation} \label{Mk}
  (M_k\gamma )(A):=\gamma(A\times X^{N-k}) \mbox{ for every Borel subset $A$ of $X^k$}
\end{equation}
(with the convention $M_N\gamma = \gamma$). 

We denote by $C_b(X^N)$ the space of bounded and continuous functions on $X^N$, 
and by $\calS_N$ the group of permutations of $\{1, \ldots, N\}$. For $\gamma \in \PP(X^N)$ and $\sigma \in \calS_N$, the measure $\gamma^{\sigma} \in \PP(X^N)$ is defined by
\[\int_{X^N} \varphi \mbox{d} \gamma^{\sigma}=\int_{X^N} \varphi(x_{\sigma(1)}, \ldots, x_{\sigma(N)}) \mbox{d} \gamma(x_1, \ldots, x_N)\]
for every test-function $\varphi \in C_b(X^N)$. A measure $\gamma \in \PP(X^N)$ is called symmetric if $\gamma=\gamma^{\sigma}$ for every $\sigma \in \calS_N$. If $\gamma \in \PP(X^N)$ is arbitrary, its symmetrization $S_N \gamma$ is given by
\begin{equation}\label{SymNOp}
S_N \gamma :=\frac{1}{N!} \sum_{\sigma \in S_N} \gamma^{\sigma}.
\end{equation}
The symmetrization operator $S_N \, : \, \gamma \mapsto S_N\gamma$ is a linear projection operator on $\PP(X^N)$, i.e. it maps $\PP(X^N)$ linearly into itself and satisfies $(S_N)^2=S_N$; and $\gamma$ is symmetric if and only if $S_N\gamma=\gamma$. The set of symmetric $N$-plans is denoted by  $\PPs(X^N)$:
\begin{equation}\label{defsymplans}
\PPs(X^N):=\{\gamma \in \PP(X^N) \; : \; \gamma=S_N \gamma \}.
\end{equation}

We shall  use the notation $\#$ to denote the push-forward measure, that is, given two Polish spaces $Y$ and $Z$, a Borel map $T$ from $Y$ to $Z$ and a Borel probability measure  $\mu$ on $Y$, then $T_\# \mu$ is the probability measure on $Z$ defined by $T_\# \mu(B)=\mu(T^{-1}(B))$ for every Borel subset $B$ of $Z$;  equivalently, for every real-valued, continuous and bounded function $\varphi$ on $Z$:
\[\int_Z \varphi \;  \mbox{d} T_\# \mu=\int_Y \varphi \circ T \;  \mbox{ d} \mu.\]

We recall the following definition from \cite{FMPCK13}. 

\begin{defi} \label{DefNrep} For $N\in\N$ and $k\in\{1,...,N\}$, 
a $k$-plan $\mu_k\in \PP(X^k)$ is said to be $N$-representable if it is the $k$-point marginal of a symmetric $N$-plan, that is to say if there exists $\gamma \in \PPs(X^N)$ such that $\mu_k = M_k \gamma$. We denote by $\PPnr(X^k)$ the set of  $N$-representable $k$-plans, i.e.:
\[\PPnr(X^k)=\{ M_k S_N \tilde{\gamma}  \; : \; \tilde{\gamma} \in \PP(X^N)\} = \{M_k \gamma \; : \; \gamma\in\PPs(X^N)\}.\]
\end{defi}

In probabilistic terms,  a symmetric $N$-plan $\gamma\in \PPs(X^N)$ is the law of a finite exchangeable random sequence $(Z_1, \ldots, Z_N)$ with values in $X^N$, whereas $\mu_k=M_k \gamma\in\PPnr(X^k)$  is the law of its first $k$-components  $(Z_1, \ldots, Z_k)$.

We will work with the following standard notion of convergence in $\PP(X)$ (as well as $\PP(X^k)$, $\PP(X^N)$, ...). Recall that $C_b(X)$ denotes the space of bounded continuous functions on $X$. 

\begin{defi} \label{narrow} A sequence $(\mu_\nu)_{\nu\in \N}$ of probability measures in $\PP(X)$ is said to converge narrowly to $\mu\in\PP(X)$ if
$$
     \lim_{\nu\to\infty} \int_{X} \varphi \mbox{d} \mu_\nu  =  \int_{X} \varphi \mbox{d} \mu \; \mbox{ for all }\varphi\in C_b(X). 
$$ 
\end{defi}
Thus, in applications to statistical physics where the probability measures live on a space of particle configurations, narrow convergence corresponds to convergence of bounded continuous observables. 

We note the following basic topological property of the set of $N$-representable $k$-plans.
\begin{lem} \label{L:narrowclosed}
The set $\PPnr(X^k)$ is closed under narrow convergence.
\end{lem}
{\bf Proof} Let $\{\mu_\nu\}_{\nu\in\N}$ be a narrowly convergent sequence in $\PPnr(X^k)$. Write $\mu_\nu$ as $M_k\gamma_\nu$ for $\gamma_\nu\in\PPs(X^N)$. Since $\mu_\nu=M_k\gamma_\nu$ is narrowly convergent, it is tight and hence so is $\gamma_\nu$. By Prokhorov's theorem, $\gamma_\nu$ has a narrowly convergent subsequence which converges to some $\gamma\in\PPs(X^N)$. Since $M_k$ is narrowly continuous, $\mu_\nu$ converges to $M_k\gamma\in\PPnr(X^k)$.
\\[2mm]
We recall that narrow convergence on $\PP(X)$ is metrizable. For instance one may
start from the metric $d$ on $X$, truncate it to the (topologically equivalent) bounded metric $\tilde{d}(x,y):=\min\{d(x,y),1\}$, and use the associated $1$-Wasserstein metric on $\PP(X)$ 
\begin{equation}\label{defW1}
W_1(\lambda_1,\lambda_2):=\inf_{ \theta\in \Pi(\lambda_1, \lambda_2)} \Big\{ \int_{X\times X} \tilde{d}(x,y) \mbox{ d} \theta(x,y) \Big\} \mbox{ for all } (\lambda_1, \lambda_2)\in \PP(X)^2,
\end{equation} 
 where $\Pi(\lambda_1,\lambda_2)$ is the set of transport plans between $\lambda_1$ and $\lambda_2$, i.e., the set of Borel probability measures on $X\times X$ having $\lambda_1$ and $\lambda_2$ as marginals. Then the (bounded) metric $W_1$ metrizes narrow convergence on $\PP(X)$ (that is, $\mu_\nu$ converges narrowly to $\mu$ if and only if $W_1(\mu_\nu,\mu)$ tends to zero) and $(\PP(X), W_1)$ is itself a Polish space. 

Also we recall the definition of the total variation distance between two signed measures $\mu$ and $\nu$ on $X$:
 \begin{equation}\label{defTV}
 \Vert \mu-\nu\Vert_{\TV}:=\sup \Big\{ \vert \mu(A)-\nu(A)\vert \; : \; A \mbox{ Borel subset of $X$}\Big\}.
 \end{equation}

\section{$\frac{1}{\rm N}$-quantized probability measures} \label{sec-quanti}

An important role will be played by the set of $\frac{1}{N}$-quantized probability measures on the Polish space $X$, 
\begin{equation} \label{quanti1}
   \PP_{\frac{1}{N}}(X) := \left\{\frac{1}{N} \sum_{i=1}^N \delta_{x_i} \, : \, x_1, \ldots, x_N \in X^N \mbox{ (not necessarily distinct)} \right\} .
\end{equation}
It is easy to see that this set can also be written as
\begin{equation} \label{quanti}
 \PP_{\frac{1}{N}}(X) =\left\{\lambda \in \PP(X) \; : \; \lambda(A) \in \left\{0,\frac{1}{N},...,1\right\}  \mbox{ for every  Borel subset $A$ of } X\right\}.
\end{equation}
In the special case of finite state spaces $X$, this set was introduced -- and utilized to parametrize extremal $N$-representable measures -- in \cite{GFDV}. Let us collect two basic properties of this set which hold for general state spaces.
\begin{lem} \label{L:closed} (Quadratic constraint characterization and closedness of $\frac{1}{N}$-quantized probability measures) \\
a) $\lambda\in\PP(X)$ belongs to $\PP_{\frac{1}{N}}(X)$ if and only if
\begin{equation}\label{caractquantized}
\lambda \otimes \lambda -\frac{2k+1}{N} \id^{\otimes 2}_\# \lambda \ge - \frac{k(k+1)}{N^2}  \; \mbox{, for } k=0, \ldots, N-1. 
\end{equation}
b) $\PP_{\frac{1}{N}}(X)$ is closed under narrow convergence.
\end{lem}

\vspace*{1mm}

\noindent
In formula \eqref{caractquantized} and in the sequel, $\lambda^{\otimes \ell}$ denotes the $\ell$-fold tensor product of $\lambda$ with itself and $\id^{\otimes \ell}_\# \lambda$ is defined by
\[\int_{X^{\ell}}  \varphi \; \mbox{d} \id^{\otimes \ell}_\# \lambda:=\int_X \varphi(x, \ldots, x) \mbox{d} \lambda(x), \; \mbox{ for all } \varphi\in C_b(X^\ell).\]
\\[2mm]
{\bf Proof} a): The nontrivial implication is that \eqref{caractquantized} implies that $\lambda\in\PP_{\frac{1}{N}}(X)$. Let $A$ be any Borel subset of $X$. Applying the measure on the left hand side of \eqref{caractquantized} to $A\times A$ gives $f(\lambda(A))\ge 0$, where $f$ is the scalar function $f(t) = t^2 - \frac{2k+1}{N}t + \frac{k(k+1)}{N^2}=(t-\frac{k}{N})(t-\frac{k+1}{N})$. But $f$ is negative precisely in the open interval $(\frac{k}{N},\frac{k+1}{N})$, whence $\lambda(A)$ does not lie in this interval. Since this holds for all $k=0,...,N-1$, it follows that $\lambda$ belongs to the set \eqref{quanti}. \\[1mm]
b): The maps $\lambda \in \PP(X) \mapsto \lambda^{\otimes 2}\in \PP(X^2)$ and $\lambda \in \PP(X) \mapsto  \id^{\otimes 2}_\# \lambda  \in \PP(X^2)$ are continuous with respect to narrow convergence, hence so is the left hand side of \eqref{caractquantized}. The assertion now follows from a). 

\section{Extremal $N$-representable $k$-plans on finite state spaces}\label{sec-rep} 
  
Throughout this section, we assume $N\ge 2$ and restrict our attention to a finite state space $X$ consisting of $\ell$ distinct points, 
  \begin{equation} \label{Xell}
X  = \{a_1, \ldots, a_\ell \}.  
  \end{equation}

\subsection{Extreme points}

Our goal is to describe the geometry of the convex set of $N$-representable $k$-plans on the finite state space $X$, i.e., $\PPnr(X^k)$. This set is a compact polyhedron in a finite-dimensional vector space and  therefore coincides, by Minkowski's theorem (see e.g. \cite{Ho94}), with the convex hull of its extreme points. Therefore, classifying the extreme points is one way to characterize the geometry of the object. We recall that a point $x$ in a convex set $K$ is an extreme point if, whenever $x=\alpha x_1 + (1-\alpha)x_2$ for some $x_1$, $x_2\in K$ and some $\alpha\in(0,1)$, we have that $x_1=x_2=x$. The set of extreme points of $K$ will be denoted $\ext{K}$. 

In \cite{GFDV} the extremal $N$-representable $k$-plans are determined in the case $k=2$ and $k=N$. They correspond exactly to the symmetrized Dirac measures respectively their two-point marginals:
\begin{thm} \label{theo1-extreme} {\rm \cite{GFDV}} {\rm a)} A measure $\mu$ on $X^N$ is an extreme point of $\PPs(X^N)$ if and only if it is of the form
\begin{equation} \label{k=N}
   S_N \delta_{a_{i_1},\ldots,a_{i_N}} \hspace{5mm} \mbox{ for some }  1 \leq i_1 \leq \ldots \leq i_N \leq \ell. 
\end{equation}
Moreover different index vectors $(i_1,...,i_N)$ with $1\le i_1\le \ldots\le i_N\le \ell$ yield different extreme points. \\[1mm]
b) A measure $\mu$ on $X^2$ is an extreme point of $\PPnr(X^2)$ if and only if it is of the form
\begin{equation} \label{k=2}
   M_2 S_N \delta_{a_{i_1},\ldots,a_{i_N}} \hspace{5mm} \mbox{ for some }  1 \leq i_1 \leq \ldots \leq i_N \leq \ell. 
\end{equation}
c) Moreover the marginal maps $M_2 \, : \, \ext(\PPs(X^N)) \to \ext(\PPnr(X^2))$ and 
$M_1 \, : \, \ext(\PPs(X^N)) \to \calP_{\frac{1}{N}}(X)$ 
are bijections. 
\end{thm}
Here a) and the fact that the set of measures in \eqref{k=2} contains the set $\ext(\PPnr(X^2))$ of extremal $N$-representable two-plans is easy to see, but the reverse inclusion and the bijectivity of $M_2$ between extremal symmetric $N$-plans and extremal $N$-representable two-plans is nontrivial; geometrically it says that none of the corners of the high-dimensional polytope $\PPs(X^N)$ is mapped into the interior (or face interior or edge interior) of the low-dimensional polytope $\PPnr(X^2)$ by the highly non-injective marginal map $M_2$. Using this nontrivial fact it is easy to extend Theorem \ref{theo1-extreme} to an arbitrary choice of $k \in \{2,\ldots,N \}$.
\begin{thm} \label{theo-extremeabstract}
A measure $\mu$ on $X^k$ is an extreme point of $\PPnr(X^k)$ if and only if it is of the form 
\begin{equation} \label{kextremalabs}
 M_k S_N \delta_{a_{i_1}, \ldots, a_{i_N}} \hspace{5mm}\mbox{for some } 1 \leq i_1 \leq \ldots \leq i_N \leq \ell.
\end{equation}
Moreover the marginal map $M_k\, : \, \ext(\PPs(X^N)) \to \ext(\PPnr(X^k))$ is a bijection.
\end{thm}
\begin{proof} We will abbreviate $\PPs(X^N)=\PPs$, $(i_1,\ldots,i_N)=i$, $\{(i_1,\ldots,i_N)\, : \, 1\le i_1\le\ldots\le i_N\le\ell\}=\calI$. By the definition of $N$-representability, $\PPnr(X^k) = M_k\PPs$. Using, in order of appearance, this fact, the linearity of $M_k$, and Theorem \ref{theo1-extreme} a), we have
\begin{equation} \label{1stinclu}
  \ext(\PPnr(X^k)) = \ext(M_k\PPs) \subseteq M_k \ext(\PPs) = \{M_kS_N\delta_{a_{i_1}, \ldots, a_{i_N}}\, : \, i\in\calI \}.
\end{equation}
To establish the reverse inclusion it suffices to show that the number of elements of the set on the left is bigger or equal that of the set on the right. By the fact that $M_2=M_2M_k$ and the linearity of $M_2$,
$$
     \ext(M_2\PPs) = \ext( M_2 M_k \PPs) \subseteq M_2 \ext(M_k\PPs) = M_2 \ext(\PPnr(X^k))
$$
and consequently $|\ext (M_2\PPs)| \le |\ext(\PPnr(X^k))|$, where $| \cdot |$ denotes the number of elements of a set. Combining this inequality with the bijectivity property of $M_2$ in Theorem \ref{theo1-extreme} b) and Theorem \ref{theo1-extreme} a) yields
\begin{eqnarray*}
  \hspace*{-2mm} |\ext (\PPnr(X^k))| &\ge & |\ext (M_2\PPs)| = |\ext(\PPs)| = |\{S_N\delta_{a_{i_1}, \ldots, a_{i_N}}\!\! : \, i\in\calI\}| \\
  & \ge &
|\{M_kS_N\delta_{a_{i_1}, \ldots, a_{i_N}}\!\! : \, i\in\calI \}|.   
\end{eqnarray*}   
\end{proof}
In \cite{GFDV} it was established that for the set \eqref{Xell} consisting of $\ell$ distinct points, the cardinality of $\calP_{\frac{1}{N}}(X)$ -- and hence, by Theorem \ref{theo1-extreme} b), the number of extreme points of $\PPs(X^N)$ -- equals ${N+\ell-1}\choose{\ell-1}$. Now the following corollary is an immediate consequence of Theorem \ref{theo-extremeabstract}. 
\begin{coro}
 For any $k \in \{2,\ldots,N\}$, $\PPnr(X^k)$ has ${N+\ell-1}\choose{\ell-1}$ extreme points. 
\end{coro} 
 
Combining the isomorphisms $M_k$ and $M_1$ from Theorems \ref{theo-extremeabstract} respectively \ref{theo1-extreme} shows that the extreme points of $\PPnr(X^k)$, i.e. the $k$-plans of form \eqref{kextremalabs}, can be uniquely recovered from their one-point marginals $\tfrac{1}{N}(\delta_{a_{i_1}}+\ldots+\delta_{a_{i_N}})\in\PP_{\frac{1}{N}}(X)$. But the above abstract reasoning does not provide a convenient formula for the recovery map. This issue is dealt with in the next section.  

\subsection{A universal polynomial formula for extreme points in terms of their one-point marginals} \label{sec:univpoly}
 
Our aim now is to derive an explicit polynomial formula for the extremal measures \eqref{kextremalabs} in terms of their one-point marginals. In order to do so we consider any extremal $N$-representable $k$-plan
 \begin{equation} \label{defmuk}
 \mu_k:= M_k S_N \delta_{x_1, \ldots, x_N}
 \end{equation} 
for $(x_1, \ldots, x_N) \in X^N$.

In \cite{GFDV} it was shown that, in the case of $k=2$, $\mu_2$ can be expressed explicitly as 
 \begin{equation}\label{mu2}
 \mu_2 =\frac{N}{N-1} \lambda^{\otimes 2} -\frac{1}{N-1} \id^{\otimes 2}_\# \lambda,
 \end{equation} 
 where 
 \begin{equation}\label{lambdaempir}
  \lambda := M_1\mu_k =\frac{1}{N} \sum_{i=1}^N \delta_{x_i} \in \PP_{\frac{1}{N}}(X)
 \end{equation} 
 is the one-point marginal of $\mu_k$. 
 (Recall the notation $\lambda^{\otimes\ell}$ and $\id^{\otimes\ell}_\#\lambda$ for the $\ell$-fold tensor product of $\lambda$ with itself respectively the push-forward of $\lambda$ under the $\ell$-fold cartesian product of the identity; see the end of Section 2.)
 A similar computation for $k=3$ gives 
\begin{equation}\label{mu3}
\mu_3=\frac{N^2}{(N-1)(N-2)} \left[\lambda^{\otimes 3} -\frac{3}{N} S_3\Big((\id^{\otimes 2}_\# \lambda) \otimes \lambda\Big)+\frac{2}{N^2} \id^{\otimes 3}_\# \lambda\right].
\end{equation}
(For a justification of \eqref{mu3} using our general results see the examples below Theorem \ref{T:expan}.)  In view of \eqref{mu2} and \eqref{mu3}, it is natural to look for a similar \textit{polynomial} of degree $k$ in $\lambda$ expression of $\mu_k$, consisting of a mean field term $\lambda^{\otimes k}$ and corrections of order $\tfrac{1}{N^j}$ for $j=1,\ldots,k\! -\! 1$. 
As turns out, the $j^{th}$ order correction is related to the partitions of the number $j$. 
\begin{defi}
Let $\N=\{1,2,3,...\}$ denote the set of positive integers. A partition of $j\in\N$ of length $n\in\N$ is a vector $\bfp=(p_1,...,p_n)\in\N^n$ such that $\sum_{i=1}^n p_i=j$, $p_1\ge \ldots \ge p_n$. For any partition $\bfp$ we denote its length by $n(\bfp)$. 
\end{defi}
\noindent
For example, the partitions of $4$ are 
\begin{align*}
      &1+1+1+1 \\
      &2+1+1 \\
      &2+2 \\
      &3+1 \\
      &4.
\end{align*} 
This corresponds in the above notation to $\bfp=(1,1,1,1)\in \N^4$, $\bfp=(2,1,1)\in\N^3$, $\bfp=(2,2)\in \N^2$, $\bfp=(3,1)\in\N^2$, and $\bfp=4\in\N$. 
\begin{thm} \label{T:expan} Let $N\ge 2$, $k\in\{2,\ldots,N\}$. Any extremal $N$-representable $k$-plan $\mu_k$ (see \eqref{defmuk}) can be written in terms of its one-point marginal $\lambda$ (see \eqref{lambdaempir}) as 
\begin{equation} \label{expan}
   \mu_k = \frac{N^{k-1}}{ \prod_{i=1}^{k-1}(N-i)} \left[ \lambda^{\otimes k} + \sum_{j=1}^{k-1}
            \frac{(-1)^j}{N^j} S_k P_j^{(k)}(\lambda ) \right ] =: F_{N,k}(\lambda )
\end{equation}
where for $j=1,...,k-1$
\begin{equation} \label{monomial}
   P_j^{(k)}(\lambda )= \!\!\! \sum_{\substack{\bfp=(p_1,...p_{n(\bfp)}) \mbox{\scriptsize partition} \\ \mbox{\scriptsize of }j\mbox{\scriptsize with }j+n(\bfp)\le k}} \!\!\!
     d^{(k)}_\bfp \, \id^{\otimes(p_1+1)}_\# \!\! \lambda \;
      \otimes \ldots \otimes \id^{\otimes(p_{n(\bfp)}+1)}_\# \!\! \lambda \; 
      \otimes \lambda^{\otimes(k-j-n(\bfp))}
\end{equation}
with positive coefficients $d_\bfp^{(k)}$ given by
\begin{equation} \label{coeffP}
   d^{(k)}_\bfp = \frac{k!}{(k-j-n(\bfp))!} \; \prod_{i=1}^{n(\bfp)} \frac{1}{p_i+1} \; 
   \prod_{q\in\mbox{\scriptsize Ran}\, \bfp} \frac{1}{( | \bfp^{-1}(q)|)! } \;\;\; .  
\end{equation}
Moreover the coefficients satisfy the sum rule
\begin{equation}\label{coeff}
     \sum_{\substack{\bfp\, \mbox{\scriptsize partition of }j \\ \mbox{\scriptsize with }j+n(\bfp)\le k}} \!\!\!
     d^{(k)}_\bfp \; = \; 
     \!\!\! \sum_{1\le i_1<\ldots <i_j\le k-1}\!\!\! i_1\!\cdot\! ... \!\cdot \! i_j \;\; =: c_j^{(k)} \;\;\; (j=1,\ldots ,k-1).
\end{equation}
In particular, $||P_j^{(k)}(\lambda)||_{TV} = \int_{X^k} dP_j^{(k)}(\lambda) = c_j^{(k)}$.
\end{thm}
 
In the last term in eq.~\eqref{coeffP}, a partition $\bfp$ is viewed as a map from the set of its component indices to $\N$; the range $\mbox{Ran}\, \bfp$ of this map is the set of values taken by the components, and $| \bfp^{-1}(q)|$ denotes the number of components with value $q$. For example, for $\bfp=(3,1,1)$ and $q=1$, $| \bfp^{-1}(q)|=2$. The factor  $(| \bfp^{-1}(q)|)!$ in the denominator says that a partition with many repeat components contributes much less than a partition with few repeat components.  

Some remarks are in order.

1) The first term in expression \eqref{expan} for the extreme points is a mean field term and the remaining terms are correlation corrections. We emphasize that the $P_j^{(k)}$ are independent of $N$ and hence the correlation corrections form a {\it finite series} in inverse powers of $N$.

2) The $P_j^{(k)}$ are polynomials of degree $k-j$ in $\lambda$. 

3) Due to the presence of the signs $(-1)^j$, it is far from trivial that $F_{N,k}(\lambda)$ is a nonnegative measure; but it must be, e.g. because the left hand side of \eqref{expan} equals \eqref{defmuk}. Nonnegativity relies on a subtle interplay between the explicit coefficients in Theorem \ref{T:expan} and the quantization condition $\lambda\in\calP_{\frac{1}{N}}(X)$, and does not hold for arbitrary $\lambda\in\calP(X)$. 

4) The coefficients $c_j^{(k)}=\sum i_1\cdot \ldots \cdot i_j$ introduced in \eqref{coeff} which measure the total mass of the $j^{th}$-order correction to independence are related to the well-known Stirling numbers, particularly the (absolute) Stirling numbers of the first kind. For given natural numbers $q,r \in \mathbb{N}\cup\{ 0 \}$ with $r\leq q$ the corresponding absolute Stirling number of the first kind $s(q,r)$ gives the number of permutations of $\{1, \ldots, q\}$ that decompose into $r$ cycles, with the convention that $s(q,r)$ is zero when exactly one of $q$ and $r$ is zero and that $s(0,0)=1$. 
From well-known expressions for Stirling numbers one can see that the following holds
\begin{equation*}
c_j^{(k)}= s(k,k-j),
\end{equation*}
that is to say the present coefficients $c_j^{(k)}$ equal the number of permutations of $\{1,\ldots,k\}$ that decompose into $k-j$ cycles. For more information about Stirling numbers 
%
we refer the interested reader to \cite{Co74}. 
%

5) By combining Theorem \ref{T:expan}, Theorem \ref{theo-extremeabstract}, and the isomorphism property of $M_1$ from Theorem \ref{theo1-extreme}, we immediately obtain: 

\begin{coro} \label{C:extfinite} A measure $\mu_k$ on $X^k$ is an extreme point of $\PPnr(X^k)$ if and only if it is of the form $\mu_k=F_{N,k}(\lambda)$ for some $\lambda\in\calP_{\frac{1}{N}}(X)$, with $F_{N,k}$ given by \eqref{expan}--\eqref{coeffP}. Moreover $F_{N,k}(\lambda)$ has one-point marginal $\lambda$, and $j$-point marginal $F_{N,j}(\lambda)$ for $2\le j\le k-1$.
\end{coro}

We now write out the universal polynomials $F_{N,k}$ explicitly for small $k$.
   
\vspace*{2mm}

   \begin{scriptsize}
   
{\bf Example: k=2.} The finite sum over $j$ in \eqref{expan} reduces to a single term for $j=1$, $-\frac{1}{N}S_2P_1^{(2)}(\lambda)$, and $P_1^{(2)}(\lambda)$ consists of a single term associated with the only partition $\bfp=1$ of $1$, $id_\#^{\otimes 2}\lambda$. Consequently
$$
   F_{N,2}(\lambda) = \frac{N}{N-1} \left[\lambda^{\otimes 2} - \frac{1}{N} \id^{\otimes 2}_\# \lambda\right].
$$
This expression agrees with \eqref{mu2}, and so Theorem \ref{T:expan} recovers \cite{GFDV} Theorem 2.1.   
\vspace*{2mm}
   
{\bf Example: k=3.} The finite series in \eqref{expan} runs from $j=1$ to $j=2$, and for these two values of $j$, the partitions contributing to the sum in \eqref{monomial} are the partitions $\bfp$ of $j$ satisfying $j+n(\bfp)\le 3$. These partitions and the associated coefficient $d_\bfp$ given by \eqref{coeffP} are 
\begin{center}\begin{tabular}{l | l | l | l }
     j & Partitions of $j$ with $j+n(\bfp)\le 3$ & our notation: $\bfp=$ &  coefficient $d_\bfp=$  \\ \hline
     1 & 1                                    & 1            &  3 \\ \hline
     2 & 2                                    & 2            &  2 \\                         
\end{tabular}\end{center}
and consequently 
$$
  F_{N,3}(\lambda) = \frac{N^2}{(N-1)(N-2)} \left[\lambda^{\otimes 3} -\frac{3}{N} S_3\Big((\id^{\otimes 2}_\# \lambda) \otimes \lambda\Big)+\frac{2}{N^2} \id^{\otimes 3}_\# \lambda\right].
$$  
   
{\bf Example: k=4.} By formulae \eqref{expan}--\eqref{coeffP}, the partitions $\bfp$ of $j$ contributing to the $j^{th}$ order correction are:
\begin{center}\begin{tabular}{l | l | l | l }
     j & Partitions of $j$ with $j+n(\bfp)\le 4$ & our notation: $\bfp=$ &  coefficient $d_\bfp=$  \\ \hline
     1 & 1                                    & 1            &  6 \\ \hline
     2 & 2                                    & 2            &  8 \\
       & 1+1                                  & (1,1)        &  3 \\ \hline
     3 & 3                                    & 3            &  6 \\
\end{tabular}\end{center}
and consequently
\begin{align*}
   F_{N,4}(\lambda ) = & \frac{N^3}{(N-1)(N-2)(N-3)} \left[ 
     \lambda^{\otimes 4} 
   - \frac{6}{N} S_4 \id^{\otimes 2}_{\#}\!\lambda \, \otimes \lambda^{\otimes 2} 
   + \frac{8 \, S_4\id^{\otimes 3}_\# \!\lambda \otimes \, \lambda 
         + 3\, S_4 \id^{\otimes 2}_\# \!\lambda \otimes \, \id^{\otimes 2}_\# \!\lambda}{N^2} \right.  \\
         & \left. - \frac{6}{N^3} \id^{\otimes 4}_\# \!\lambda 
    \right].
\end{align*}

{\bf Example: k=5.} The contributing partitions are
\begin{center}\begin{tabular}{l | l | l | l }
     j & Partitions of $j$ with $j+n(\bfp)\le 5$ & our notation: $\bfp=$ &  coefficient $d_\bfp=$  \\ \hline
     1 & 1                                    & 1            &  10 \\ \hline
     2 & 2                                    & 2            &  20 \\
       & 1+1                                  & (1,1)        &  15 \\ \hline
     3 & 3                                    & 3            &  30 \\
       & 2+1                                  & (2,1)        &  20 \\ \hline
     4 & 4                                    & 4            &  24 \\
\end{tabular}\end{center}
and so
\begin{align*}
   F_{N,5}(\lambda ) = & \frac{N^4}{(N-1)(N-2)(N-3)(N-4)} \left[ 
     \lambda^{\otimes 5} 
   - \frac{10}{N} S_5 \id^{\otimes 2}_{\#}\!\lambda \, \otimes \lambda^{\otimes 3} 
   + \frac{20 \, S_5\id^{\otimes 3}_\# \!\lambda \otimes \, \lambda^{\otimes 2} 
         + 15\, S_5 \id^{\otimes 2}_\# \!\lambda \otimes \, \id^{\otimes 2}_\# \!\lambda \otimes \lambda}{N^2} \right. \\
         & \left. - \frac{30 \, S_5\id^{\otimes 4}_\# \!\lambda \otimes \, \lambda 
         + 20\, S_5 \id^{\otimes 3}_\# \!\lambda \otimes \, \id^{\otimes 2}_\# \!\lambda}{N^3} 
    + \frac{24}{N^4} \id^{\otimes 5}_\# \!\lambda 
    \right ].
\end{align*}

{\bf Example: j=1 and j=k-1.} In these cases only one partition of $j$ satisfies $j+n(\bfp)\le k$ and formula \eqref{coeffP} for the coefficient $d_\bfp$ becomes particularly simple:
\begin{center}\begin{tabular}{l | l | l | l }
     j & Partitions of $j$ with $j+n(\bfp)\le k$ & our notation: $\bfp=$ &  coefficient $d_\bfp=$  \\ \hline
     1 & 1                                    & 1            &  $\tfrac{k(k-1)}{2}$ \\ \hline
     k-1 & k-1                                & k-1          &  (k-1)! \\
\end{tabular}\end{center}
It follows that the polynomials describing the first-order respectively order-(k-1) contribution to $F_{N,k}(\lambda)$ are
\begin{equation*}
     P_1^{(k)}(\lambda ) = \frac{k(k-1)}{2} \, \id^{\otimes 2}_{\#}\!\lambda \, \otimes \lambda^{\otimes (k-2)}  , \;\;\;\;\;\;\; P_{k-1}^{(k)}(\lambda ) = (k-1)! \, \id^{\otimes k}_\# \!\! \lambda .
\end{equation*}
\vspace*{1mm}

\end{scriptsize}
We now discuss the error when truncating the finite series in \eqref{expan}. 
Retaining only the mean-field term gives
\begin{equation} \label{err}
   F_{N,k}(\lambda) = \frac{N^{k-1}}{ \prod_{j=1}^{k-1}(N-j)} \left( \lambda^{\otimes k} + \eps_{N,k}(\lambda)\right) \mbox{ with }  \Vert \eps_{N, k} (\lambda)\Vert_{\TV} \le \frac{C_k}{N}
\end{equation}
and keeping the first $p$ correction terms ($p\in\{1,...,k-2\}$) we have
\begin{align} 
   F_{N,k}(\lambda) = & \frac{N^{k-1}}{ \prod_{j=1}^{k-1}(N-j)} \left[ \lambda^{\otimes k} + \sum_{j=1}^p  \frac{(-1)^j}{N^j} S_k \, P_j^{(k)}(\lambda ) + 
\eps_{N,k,p}(\lambda)\right] \nonumber \\ & \; \mbox{ with }  \Vert \eps_{N, k,p} (\lambda)\Vert_{\TV} \le \frac{C_k}{N^{p+1}}, \label{errp}
\end{align}
with constants $C_k$ independent of $N$ and $p$. For example, to give explicit values,
\begin{equation} \label{Ck}
   C_k=\sum_{j=1}^{k-1}c_j^{(k)}
\end{equation}
will do. Moreover the coefficients $d_\bfp^{(k)}$ and hence the $C_k$ are independent of the size $\ell$ of the finite state space.   
Thus, for $k$ fixed and any $N\ge 2$, retaining only the first $p$ correlation terms captures the extreme points up to an error which decreases {\it exponentially} in $p$, the rate being {\it uniform} in the size of the finite state space and {\it improving} logarithmically with $N$. 
\\[2mm]
Before proving Theorem \ref{T:expan} let us give a quick heuristic derivation of the formulae for the coefficients $c_j^{(k)}=\int dP_j^{(k)}(\lambda)$ which give the total mass of the $j^{th}$ order correction to independence for extremal $N$-representable $k$-plans. Expressions \eqref{mu2} and \eqref{mu3} suggest to try the ansatz 
\begin{equation}\label{expanshort}
      \mu_k =  \frac{N^{k-1}}{ \prod_{j=1}^{k-1}(N-j)} \left[ \lambda^{\otimes k} + \sum_{j=1}^{k-1}
            \frac{(-1)^j}{N^j} c_j^{(k)} \nu_j \right]
\end{equation}
with normalized measures $\nu_j$ (i.e. $\int d\nu_j=1$) and a priori unknown but $N$-independent coefficients $c_j$. Consider for example $k=4$. Integrating over $X^k$, using that $\mu_k$ and the $\nu_j$ are normalized, and multiplying both sides by the product $\prod_{j=1}^{k-1} (N-j)$ gives
\begin{equation} \label{ceq}
   N^3(1-\frac{c_1}{N} + \frac{c_2}{N^2} - \frac{c_3}{N^3}) = (N-1)(N-2)(N-3).
\end{equation}
Expanding the right hand side into powers of $N$ gives
$$
  (N-1)(N-2)(N-3) = N^3 - (1+2+3)N^2 + (1\cdot 2 + 1\cdot 3 + 2\cdot 3)N - (1\cdot 2\cdot 3)
$$ 
so equating coefficients yields $c_1=\sum_{1\le i\le 3} i \, (=6)$, $c_2=\sum_{1\le i<j\le 3}ij \, (=11)$, $c_3=\sum_{1\le i<j<k\le 3} ijk \, (=6)$, i.e. the asserted formulae for the $c_j$. Extending this heuristic argument to general $k$ is straightforward. 

Of course this argument is not a proof because it rests on the (as yet unjustified) ansatz \eqref{expanshort} with $N$-independent coefficients. This ansatz is a corollary of the more detailed result \eqref{expan}--\eqref{coeff} to whose proof we now turn.

We begin by eliminating the high-dimensional space $\PPs(X^N)$ which appears in \eqref{defmuk}. 

\begin{lem} \label{L:elid}
Any extremal $N$-representable $k$-plan $\mu_k$ given by \eqref{defmuk} can be written as 
\begin{equation}\label{elid}
\mu_k=\frac{(N-k)!}{N!} \sum_{\substack{(m_1, \ldots, m_k)\in \{ 1,...,N\}^k, \\ \mbox{\scriptsize pairwise distinct}}} \delta_{x_{m_1}\ldots x_{m_k}}.
\end{equation}  
\end{lem}

\begin{proof}
Proceeding as in \cite{GFDV}, more specifically the proof of Lemma 2.1, we rewrite $\mu_k$ by plugging in the definition of the symmetrization operator $S_N$ and conditioning the sum over all permutations $\sigma\ : \, \{1,...,N\}\to\{1,...,N\}$ on the values on the first $k$ integers: 
 \[
  \begin{split}
 N! \mu_k &= M_k\sum_{\sigma \in S_N} \delta_{x_{\sigma(1)} \ldots x_{\sigma(N)}}  \\
 &=  M_k\sum_{\substack{(m_1,\ldots,m_k)\in \{ 1,...,N \}^k \\ \mbox{\scriptsize pairwise distinct} }}\sum_{\substack{\sigma \in S_N \\ \sigma(1) = m_1,...,\sigma(k)=m_k}} \delta_{x_{m_1} \ldots x_{m_k}x_{\sigma(k+1)} \ldots x_{\sigma(N)}} \\
  &= (N-k)! \sum_{\substack{(m_1, \ldots , m_k) \in \{ 1,...,N\}^k \\ \mbox{\scriptsize pairwise distinct}}} \delta_{x_{m_1} \ldots x_{m_k}}.
   \end{split}
 \] 

\end{proof} 
 
To establish Theorem \ref{T:expan} we will proceed by induction over $k$. The next lemma gives a deceptively simple recursion formula for extremal $N$-representable $k$-plans. Just like \eqref{elid}, it hides the inverse power series structure \eqref{expan} and the combinatorial complexity of the coefficients \eqref{coeffP} by expanding the symmetric plan $\mu_k$ in a non-symmetric basis of delta functions, leading to many terms with identical symmetrization. 

\begin{lem} \label{L:recur}
Let $N\ge 2$, and consider the $k$-plans $\mu_k$ ($k=2,...,N$) defined by \eqref{defmuk} for fixed $(x_1,...,x_N)\in X^N$. Then for $k=1,...,N-1$ and $\lambda=M_1\mu_k$,
\begin{equation}\label{recurmu}
 \mu_{k+1}= \frac{N}{N-k} \mu_k \otimes \lambda - \frac{1}{N-k}\sum_{j=1}^k {R_j}_\# \mu_k
\end{equation}
where $R_j$ : $X^k \to X^{k+1}$ is given by $R_j(z_1, \ldots, z_k):=(z_1, \ldots, z_k, z_j)$.
\end{lem}
 
\begin{proof}
We observe that $(m_1, \ldots, m_k, m_{k+1})\in\{1,...,N\}^{k+1}$ has pairwise distinct components if and only if $(m_1,...,m_k)$ has pairwise distinct components and $m_{k+1} \in \{1, \ldots,N\}\setminus\{m_1, \ldots, m_k\}$. So by inclusion-exclusion we get, using $(N-(k+1))!=(N-k)!/(N-k)$, 
\begin{align*}
 \mu_{k+1} &= \frac{(N-(k+1))!}{N!} \sum_{\substack{ (m_1,...,m_{k}) \in \{ 1,...,N\}^k \\ \mbox{\scriptsize pairwise distinct}}} \sum_{m_{k+1}\in \{ 1,...,N \}\setminus \{m_1,...,m_k \} }
               \delta_{x_{m_1}...x_{m_{k+1}}} \\
           &= \frac{(N-k)!}{N! \, (N-k)}  \Bigl( \sum_{i=1}^N \sum_{\substack{(m_1,...,m_{k}) \in \{ 1,...,N\}^k \\ \mbox{\scriptsize pairwise distinct}}} \delta_{x_{m_1}...x_{m_k}x_i} - 
           \sum_{\substack{(m_1,...,m_{k})\in \{ 1,...,N\}^k \\ \mbox{\scriptsize pairwise distinct}}} \sum_{j=1}^k \delta_{x_{m_1}...x_{m_k}x_{m_j}} \Bigr) \\
           &= \frac{N}{N-k}\mu_k\otimes\lambda \, - \, \frac{1}{N-k}\sum_{j=1}^k {R_j}_{\#} \mu_k.
\end{align*}
\end{proof}
Next we derive a non-recursive formula in terms of {\it set partitions} of $\{1,...,k\}$. To state it we need to introduce some notation. Recall that a set partition of $\{1,...,k\}$ is a set $\calP$ of pairwise disjoint nonempty subsets of $\{1,...,k\}$ (called blocks) whose union equals $\{1,...,k\}$. The set of all such partitions will be denoted $\parti_k$. For a partition $\calP\in\parti_k$, we denote by $n(\calP)$ the cardinality of $\calP$, so that $\calP=\{P_1, \ldots, P_{n(\calP)}\}$ for some $P_i\subseteq\{1,...,k\}$, and introduce the combinatorial factor
 \begin{equation}\label{defbetap}
   \beta_{\calP}:=\prod_{i=1}^{n(\calP)} ( | P_i| -1)! = \prod_{P\in\calP} ( | P| - 1)! \, . 
 \end{equation} 
Next, each partition $\calP$ induces a certain natural mapping $G_{\calP} \, : \, \calP(X) \to \calP(X^k)$. Informally, this mapping pushes, for each block $P$ of the partition $\calP$, a tensor factor $\lambda\in\calP(X)$ forward onto the diagonal of those cartesian factors $X_{i_1},...,X_{i_{| P|}}$ of the product space $X^k$ whose indices belong to $P$. More precisely, if $\calP=\{P_1, \ldots, P_{n}\}$, define $G_{\calP}(\lambda)$ by 
\begin{equation} \label{defGP}
   G_{\calP}(\lambda)(A_1\times ... \times A_k) = \prod_{P\in\calP} 
   \Bigl(\id^{\otimes | P|}_{\#} \lambda \Bigl)\Bigl(\prod_{i\in P} A_i\Bigr) \; \mbox{ for any }A_1,...,A_k\subseteq X.
\end{equation}
For instance if $k=4$ and $\calP=\{\{1,2\}, \{3,4\}\}$, $G_{\calP}(\lambda)=(\id^{\otimes 2}_\# \lambda)\otimes (\id^{\otimes 2}_\# \lambda)$, whereas if $k=5$ and $\calP=\{\{1,3\}, \{2,4,5\}\}$ then $G_{\calP}(\lambda)$ is defined by
 \[\int_{X^5} \varphi \; \mbox{ d} G_{\calP}(\lambda)=\int_{X^2} \varphi (x,y,x,y,y) \;  \mbox{d} \lambda(x) \;  \mbox{d} \lambda(y), \; \mbox{ for all } \varphi \in C_b(X^5).\] 
  
We then have the following representation formula.

\begin{prop}\label{repfinite}
Let $\mu_k$ be defined by \eqref{defmuk}, and let $\lambda$ be its one-point marginal \pref{lambdaempir}. Then 
    \begin{equation}\label{nonrecurmu}
  \mu_k = \frac{(N-k)!}{N!}  \sum_{\calP\in \parti_k} (-1)^{k-n(\calP)} N^{n(\calP)} 
  \beta_{\calP} G_{\calP}(\lambda).
    \end{equation}
\end{prop}

 \begin{proof}
For $k=1$ the assertion is obvious, and for $k=2$ it easily follows from \pref{mu2}. Let us assume that \eqref{nonrecurmu} holds for $k\le N-1$. By \eqref{recurmu}
\[
  \begin{split}
  \mu_{k+1}=\tfrac{(N-(k+1))!}{N!}(A+B), \, 
  \mbox{ with } A:=\sum_{\calP\in \parti_k} (-1)^{k-n(\calP )} N^{n(\calP )+1} \beta_{\calP} G_{\calP}(\lambda)\otimes \lambda \\
 \mbox{ and }  \; B:=\sum_{j=1}^k \sum_{\calP\in \parti_k} (-1)^{k+1-n(\calP )} N^{n(\calP )} \beta_{\calP}  {R_j}_\# G_{\calP}(\lambda).
  \end{split}
\]
Now let us partition $\parti_{k+1}$ into the two subsets $\parti_{k+1}^a$ and its complement $\parti_{k+1}^b$ where $\parti_{k+1}^a$ consists of all partitions $\calP'$ of $\{1, \ldots, k+1\}$ for which the singleton $\{k+1\}$ belongs to $\calP'$. Thus $\calP'\in \parti_{k+1}^a$ if and only if it can be written as $\calP\cup\{\{k+1\}\}$ with $\calP\in \parti_k$. Note then that $n(\calP')=n(\calP)+1$, $\beta_\calP=\beta_{\calP'}$ and $G_{\calP'}(\lambda)=G_\calP(\lambda) \otimes \lambda$. Therefore we have 
\[
\begin{split}
&\sum_{\calP'\in \parti_{k+1}^a} (-1)^{k+1-n(\calP')} N^{n(\calP')} \beta_{\calP'} G_{\calP'}(\lambda)\\
  &=\sum_{\calP\in \parti_{k}} (-1)^{k-n(\calP)} N^{n(\calP)+1} \beta_{\calP} G_{\calP}(\lambda) \otimes \lambda\\
 &=A.
\end{split}
\]
 Partitions $\calP'$ in $\parti_{k+1}^b$ are those for which $k+1$ does not form a singleton in $\calP'$. This is the same as saying that the following map $a^+_{k+1}$ from $\{(\calP, P) \, : \, \calP\in \parti_k, \; P\in\calP \}$ to $\parti_{k+1}^b$ is a bijection:  $a^+_{k+1}(\{P_1,...,P_n\}, \, P_i) := \calP' := \{ P_1,...,P_i\cup\{k+1\},...,P_n\}$ ($i=1,...,n$), or -- in label-free notation -- 
$$
    a^+_{k+1}(\calP,P) := \{ P\cup\{k+1\} \} \cup \{ Q\in\calP \, : \, Q\neq P \}.
$$
We chose the notation $a^+_{k+1}$ to emphasize the analogy with creation operators in quantum theory: the map $a^+_{k+1}$ ``creates'' an extra entry $k+1$ in some block of the partition. Note that if $\calP'=a^+_{k+1}(\calP,P)$, then $n(\calP)=n(\calP')$, $\beta_{\calP'}=\beta_{\calP} \cdot   |P|$, and $G_{\calP'}(\lambda)={R_j} _\# G_{\calP}(\lambda)$ for every $j\in P$, so that 
$$
   \beta_{\calP'} G_{\calP'}(\lambda)=\beta_\calP \sum_{j\in P} {R_j}_\#G_\calP(\lambda). 
$$
We thus have
 \[\begin{split}
&\sum_{\calP'\in \parti_{k+1}^b} (-1)^{k+1-n(\calP')} N^{n(\calP')} \beta_{\calP'} G_{\calP'}(\lambda)\\
&= \sum_{\calP\in \parti_k} \sum_{P\in\calP } (-1)^{k+1-n(\calP)} N^{n(\calP)} \beta_\calP \sum_{j\in P} {R_j}_\# G_{\calP}(\lambda )   \\
&=\sum_{j=1}^k \sum_{\calP\in \parti_k} (-1)^{k+1-n(\calP)} N^{n(\calP)} \beta_\calP \Big(\sum_{P\in\calP \; : \; j\in P} 1 \Big) {R_j}_\# G_\calP(\lambda)\\
&=\sum_{j=1}^k \sum_{\calP\in \parti_k} (-1)^{k+1-n(\calP)} N^{n(\calP)} \beta_\calP  {R_j}_\# G_\calP(\lambda)\\
&=B,
\end{split}\]
 which gives the desired expression for $\mu_{k+1}=\tfrac{(N-(k+1))!}{N!}(A+B)$.
 
  \end{proof}

It remains to match the unwieldy-to-evaluate expression \eqref{nonrecurmu} with the more explicit expansion  \eqref{expan}, \eqref{monomial}, \eqref{coeffP}. 

We begin by dealing with the fact that expression \eqref{nonrecurmu} contains many terms with identical symmetrization. Since $\mu_k$ is symmetric, applying the symmetrization operator $S_k$ to both sides gives 
\begin{equation} \label{muksymmetrized}
  \mu_k = \sum_{\calP\in\parti_k} (-1)^{k-n(\calP)} 
  \frac{N^{n(\calP)}}{N\cdot(N-1)\cdot ... \cdot (N-k+1)} \, 
  \beta_{\calP}\,  S_k\,  G_{\calP}(\lambda).
\end{equation}
If $\calP\in\parti_k$, then according to \eqref{defGP} $G_\calP(\lambda)$ pushes $n(\calP)$ factors $\lambda$ onto the $k$ cartesian factors of the product space $X^k$. The different 
$S_k G_{\calP}(\lambda)$'s which can arise from such a set partition are in bijective 
correspondence to the partitions $\bfp'$ of the number $k$, via
\begin{equation} 
\bfp'= (p'_1,...,p'_{m} )
   \longmapsto S_k \id^{\otimes p'_1}_\# \!\lambda \, \otimes \, ... \otimes \id^{\otimes p'_{m}}_\# \!\lambda . \label{bij}
\end{equation}
The set partitions $\calP\in\parti_k$ satisfying $S_kG_{\calP}(\lambda) =\,$r.h.s. of \eqref{bij} for some given partition $\bfp'=(p'_1,...,p'_m)$ of $k$ are precisely those consisting of $m$ sets $P_1,...,P_{m}$ with cardinalities $| P_i|=p'_i$ for all $i$. Let us denote their totality by $\parti_k(\bfp')$. A canonical set partition in $\parti_k(\bfp')$ is 
\begin{align*}
  & \calP(\bfp') = \Bigl\{ P_1,...,P_m \Bigr\} \; \mbox{ with} \\
  & P_1 = \{1,...,p'_1\}, \, P_2 = \{p'_1+1,...,p'_1 + p'_2\},..., \\
  & P_m=\{ p'_1+...+p'_{m-1}, ..., \underbrace{p'_1+...+p'_m}_{=k}\} .               
\end{align*}
For this set partition, as well as any other $\calP\in\parti_k(\bfp')$,
\begin{equation} \label{betaPp}
   n(\calP) = m, \;\;\; \beta_{\calP} = (p'_1-1)! \cdot ... \cdot (p'_m-1)!.
\end{equation}
To reduce \eqref{nonrecurmu} to a sum over partitions $\bfp'$ of the number $k$, it remains to determine the number $\calN(\bfp')$ of set partitions belonging to $\parti_k(\bfp')$. Let us fix 
any partition $\bfp'=(p'_1,...,p'_m)$ of $k$. First of all we note that
\begin{equation}\label{ordering}
  \calN(\bfp') = \frac{1}{\prod\limits_{q\in \mbox{\scriptsize Ran}\, \bfp'} (| \bfp' {^{-1}}(q)|)!} \, {\calN} {'} 
\end{equation}
where ${\calN} {'}$ is the number of set partitions corresponding to $\bfp'$ endowed with an ordering of the blocks such that larger blocks come before smaller ones, ${\calN}  {'}=|\{(P'_1,...,P'_{m}) \, : \, | P'_i|=p'_i \mbox{ for all }i\}|$. Here the combinatorial factor in the denominator accounts for the fact that 
any group of $b$ equal-sized blocks in a set partition admits $b!$ orderings. 
But the number $\calN'$ is straightforward to compute: choosing $P'_1$ means choosing $p'_1$ numbers out of $k$, so there are ${k\choose p'_1}$ choices; given $P'_1$, choosing $P'_2$ means choosing $p'_2$ numbers out of the remaining $k-p'_1$ numbers, yielding ${k-p'_1\choose p'_2}$ choices; and so on. It follows that   
\begin{equation} \label{N'}
  {\calN} ' = {k \choose p'_1} \cdot {k-p'_1 \choose p'_2} \cdot ... \cdot
          {k-(p'_1 + ... + p'_{m-1}) \choose p'_m}  = \frac{k!}{p'_1!\cdot ... \cdot p'_m!}.
\end{equation}
Combining \eqref{muksymmetrized}, the bijectivity of the map \eqref{bij}, \eqref{betaPp}, \eqref{ordering}, and  \eqref{N'} yields:
\begin{prop} \label{P:almostewens}
Let $\mu_k$ be any extremal $N$-representable $k$-plan (see \eqref{defmuk}), and let $\lambda$ be its one-point marginal \pref{lambdaempir}. Then 
\begin{align} 
&  \mu_k =  \sum_{p'\, \mbox{\scriptsize partition of }k}
   c_{\bfp'} S_k \id^{\otimes p'_1}_\# \!\lambda \, \otimes \, ... \otimes \id^{\otimes p'_{n(\bfp')}}_\# \!\lambda \mbox{ with } \nonumber \\
& 
  c_{\bfp'} = (-1)^{k-n(\bfp' )} \! \frac{k!}{N(N-1)...(N-k+1)} \! N^{n(\bfp' )} \frac{1}{p'_1\cdot ... \cdot p'_{n(\bfp')}} \cdot  \frac{1}{\prod\limits_{q\in \mbox{\scriptsize Ran}\, \bfp'} (| \bfp' {^{-1}}(q)|)!} .   
  \label{nonrecurmu2}
\end{align}
\end{prop}

The expression for $\mu_k$ in \eqref{nonrecurmu2} may be taken as an alternative definition of the polynomial $F_{N,k}(\lambda)$ introduced in Theorem \ref{T:expan}. 

\subsection{End of the proof of Theorem \ref{T:expan}}

The last step in the proof of Theorem \ref{T:expan} is to match the above expression with the series given in the theorem. We would like to decompose \eqref{nonrecurmu2} into terms of order $\tfrac{1}{N^j}$ (times the overall order $1$ prefactor in \eqref{expan}). To this end, we re-write the $N$-dependent factors in \eqref{nonrecurmu2} as 
$$
   \frac{1}{N\cdot(N-1)\cdot ... \cdot(N-k+1)} N^{n(\bfp')} = \frac{N^{k-1}}{(N-1) \cdot ... \cdot (N-k+1)} \, \frac{1}{N^{k-n(\bfp')}}.
$$
This together with the fact that the only partition $\bfp'$ of $k$ with $k-n(\bfp')=0$ is $\bfp'=(1,...,1)$, in which case the r.h.s. of \eqref{bij} is $\lambda^{\otimes k}$ and  $\tfrac{1}{p'_1\cdot ... \cdot p'_m} = 1$, shows that
\begin{equation} \label{muksplitorder}
\begin{split}
   \mu_k = \frac{N^{k-1}}{(N-1)...(N-k+1)} \Bigl( \lambda^{\otimes k} + \sum_{j=1}^{k-1} \frac{(-1)^j}{N^j} \sum_{\substack{\bfp' \, \mbox{\scriptsize partition of }k \\ 
   \mbox{\scriptsize with }n(\bfp') = k-j}} \frac{k!}{p'_1\cdot ... \cdot p'_{n(\bfp')}} \cdot \\ 
   \frac{1}{\prod\limits_{q\in \mbox{\scriptsize Ran}\, \bfp'} (| \bfp' {^{-1}}(q)|)!} 
   S_k \id^{\otimes p'_1}_\# \!\lambda \, \otimes \, ... \otimes \id^{\otimes p'_{n(\bfp')}}_\# \!\lambda \Bigr) .
   \end{split}
\end{equation}
The partitions $\bfp'$ of $k$ of length $n(\bfp')=k-j$ are in bijective correspondence to the partitions $\bfp$ of the number $j$ with length $n(\bfp)\le k-j$, via  
\begin{align} 
  & \bfp=(p_1,...,p_n) \mbox{ partition of }j \mbox{ of length }n\le k-j \nonumber \\
  & \longmapsto \bfp'=(p'_1,...,p'_{k-j})=\begin{cases} (p_1+1,...,p_n+1) & \mbox{if }n=k-j \\
    (p_1+1,...,p_n+1,\!\!\!\underbrace{1,...,1}_{k-j-n\,\mbox{\scriptsize times}}\!\!\! ) & \mbox{if }n<k-j.
    \end{cases} \label{bij2}
\end{align}  
Moreover, for any two partitions $\bfp$, $\bfp'$ related by \eqref{bij2}, we have 
\begin{equation} 
    S_k \id^{\otimes p'_1}_\# \!\lambda \, \otimes \, ... \otimes \id^{\otimes p'_{m}}_\# \!\lambda
  = S_k \id^{\otimes p_1+1}_\# \!\lambda \, \otimes \, ... \otimes \id^{\otimes p_n+1}_\# \!\lambda \otimes \lambda^{\otimes k-(j+n)} \label{pp'1}
\end{equation}
and
\begin{align}
   & \; p'_1\cdot ... \cdot p'_{n(\bfp')} = (p_1+1)\cdot ... \cdot (p_n+1), \;\;\;
    \; \label{nonumber} \\
   & \prod\limits_{q\in \mbox{\scriptsize Ran}\, \bfp'} (| \bfp' {^{-1}}(q)|)! 
    = \prod\limits_{q\in \mbox{\scriptsize Ran}\, \bfp} (| \bfp{^{-1}}(q)|)! \, \cdot \,    (k-(j+n))!, \label{pp'2}
\end{align}
with the last factor above accounting for the $k-(j+n)$ components with value $1$ occurring in $\bfp'$. Combining eq.~\eqref{muksplitorder}, the bijectivity of \eqref{bij2}, and identities \eqref{pp'1}--\eqref{pp'2} yields the desired expression for $\mu_k$ in \eqref{expan}-\eqref{monomial}-\eqref{coeffP}. Let us finally prove formula \pref{coeff}. Observe that the detailed representation  \eqref{expan} implies \eqref{expanshort}, where 
 each $\nu_j$ is a probability measure and the coefficients $c_j^{(k)}$ do not depend on $N$.  Taking the total mass of each side of \eqref{expanshort} we get that for every $N\geq k$, one has
\[\prod_{j=1}^{k-1} (N-j) =N^{k-1} + \sum_{j=1}^{k-1}(-1)^j c_j^{(k)} N^{k-j-1}\]
so that  $(-1)^j c_j^{(k)}$ is the coefficient of $X^{k-j-1}$ in the polynomial $\prod_{j=1}^{k-1}(X-j)$ which proves \eqref{coeff} thanks to Vieta's formulas.

\section{Extreme points and integral representation of $N$-representable $k$-plans on continuous state spaces}  \label{sec-cts}
  
Now we return to the case of a general state space, i.e. we just assume that $X$ is a Polish space. Importantly, the results in this section cover the prototypical continuous state space  $X=\R^d$. 

Recall the set $\PP_{\frac{1}{N}}(X)$ of $\frac{1}{N}$-quantized probability measures on $X$ introduced in \eqref{quanti1}, \eqref{quanti}. 

Given $\lambda\in\PP(X)$ we define $F_{N,k}(\lambda)\in\PP(X^k)$ by formulae \eqref{expan}, \eqref{monomial}, \eqref{coeffP}; we note that the expressions in these formulae make sense for general Polish spaces $X$ and general (not necessarily $\frac{1}{N}$-quantized) $\lambda\in\PP(X)$.

  \subsection{De Finetti style representation}
  
  We now state a de Finetti style representation result for $N$-representable $k$-plans.
  
  \begin{thm}\label{Choquet} Let $N\ge k \ge 2$. A measure
  $\mu_k \in \PP(X^k)$ is $N$-representable if and only if there exists $\alpha \in \PP(\PP(X))$ such that $\alpha(\PP_{\frac{1}{N}}(X))=1$ and 
  \begin{equation}\label{choquetformula}
  \mu_k=\int_{\PP_{\frac{1}{N}}(X)} F_{N,k}(\lambda) \, \mbox{d} \alpha(\lambda)
  \end{equation}
  where $F_{N,k}$ is defined by \pref{expan}--\pref{coeffP}. Moreover, if $k=N$, the measure $\alpha$ in \pref{choquetformula} is unique.
  \end{thm}
  
The meaning of \eqref{choquetformula} is that for every test function $\varphi\in C_b(X^k)$, 
\[  \int_{X^k} \varphi \, d\mu_k = \int_{\PP_{\frac{1}{N}}(X)} \Phi(\lambda)d\alpha(\lambda),\]
where
\[ \Phi(\lambda) := \int_{X^k} \varphi(x_1,...,x_k) d \Bigl( F_{N,k}(\lambda) \Bigr)(x_1,...,x_k).\]

Note that $\Phi \, : \, \lambda\mapsto\Phi(\lambda)$ is a continuous function on $\PP(X)$ endowed with the narrow topology, by the narrow-to-narrow continuity of $F_{N,k}$. Moreover the sup norm of $\Phi$ is bounded by the sup norm of $\varphi$ times the TV norm of $F_{N,k}(\lambda)$, the finiteness of the latter being clear from the definition. It follows that $\Phi\in C_b(\PP(X))$, whence $\int_{\PP_{\frac{1}{N}}(X)} \Phi(\lambda)d\alpha(\lambda)$ 
is well defined for any $\alpha\in\PP(\PP(X))$.

  Theorem \ref{Choquet} generalizes the celebrated de Finetti-Hewitt-Savage theorem from infinitely representable to finitely representable measures, or - in probabilistic language - from infinitely to finitely extendible sequences of random variables. Formally, in the limit $N\to\infty$, the domain of integration $\PP_{\frac{1}{N}}(X)$ in \eqref{choquetformula} approaches all of $\PP(X)$ and the integrand tends to the independent measure $\lambda^{\otimes k}$, recovering de Finetti (see Section \ref{sec-hs} for a rigorous account). Thus in the finitely representable case, the role of the independent measures $\lambda^{\otimes k}$ in de Finetti is taken by the universally correlated measures $F_{N,k}(\lambda)$, which contain corrections of order $N^{-j}$ for $j=1,...,k$. 
  
  The additional assumption $k=N$ for uniqueness cannot be omitted, see the next section for a simple counterexample.
   
  \begin{proof}
  Recall that in a Polish space, finitely supported probability measures are dense with respect to narrow convergence. We know from Theorem \ref{T:expan} that when $\lambda:=N^{-1} \sum_{i=1}^N \delta_{x_i}\in \PP_{\frac{1}{N}}(X)$ then $F_{N,k}(\lambda)=M_k S_N \delta_{x_{1}\ldots x_N}$ and so $F_{N,k}(\lambda)\in \PPnr(X^k)$. By convexity of $\PPnr(X^k)$, any measure of the form \pref{choquetformula} with a finitely supported probability measure $\alpha$ on $\PP_{\frac{1}{N}}(X)$ also belongs to $\PPnr(X^k)$. Now if $\alpha\in \PP(\PP_{\frac{1}{N}}(X))$ is arbitrary and $\mu_k$ is given by \pref{choquetformula}, we approximate $\alpha$ by a sequence of finitely supported measures $\alpha^n$. We now use that the map $\lambda \mapsto F_{N, k}(\lambda)$ is continuous under narrow convergence. This follows immediately from the definition \eqref{expan}--\eqref{coeffP} and the continuity of the maps $\lambda\mapsto\lambda^{\otimes j}$ and $\lambda\mapsto\id^{\otimes j}_{\#}\lambda$. Hence $\mu_k^n:=\int_{\PP_{\frac{1}{N}}(X)} F_{N,k}(\lambda) \mbox{d} \alpha^n(\lambda)$ converges narrowly to $\mu_k$ which therefore belongs to $\PPnr(X^k)$, since the latter is closed under narrow convergence (see Lemma \ref{L:narrowclosed}).

  \smallskip

  Conversely, given $\gamma\in \PP(X^N)$ and $\mu_k:=M_k S_N \gamma\in \PPnr(X^k)$, let $\gamma^n$ be a sequence of finitely supported probability measures which narrowly converges to $\gamma$, so that $\mu_k^n:=M_k S_N \gamma^n$ converges to $\mu_k$. Let us write $\gamma^n:=\sum_{j \in J_n} \gamma_j^n \delta_{\xx_j^n}$ where $J_n$ is finite and $\xx_j^n=(x_{1, j}^n, \ldots x_{N, j}^n) \in X^N$. Using Theorem \ref{T:expan} again, we know that $\mu_k^n$ can be written as
  \begin{equation}\label{approxrep}
  \mu_k^n:=\int_{\PP_{\frac{1}{N}}(X)} F_{N,k}(\lambda) \;  \mbox{d} \alpha^n(\lambda)
  \end{equation}
  with
  \begin{equation}\label{defalphan}
  \alpha^n:=\sum_{j \in J_n} \gamma_j^n \delta_{\Lambda(\xx_j^n)}=\Lambda_\# \gamma^n 
  \end{equation}
  where the map $\Lambda$ : $X^N \mapsto \PP_{\frac{1}{N}}(X)$ is defined for all $\xx:=(x_1, \ldots, x_N)\in X^N$ by
  \[\Lambda(\xx):=\frac{1}{N} \sum_{i=1}^N \delta_{x_i}.\]
  Note that $\Lambda$ is Lipschitz with respect to $W_1$; in particular it is continuous from $X^N$ to $\PP(X)$ endowed with the narrow topology. Since $\gamma^n$  converges narrowly to $\gamma$, $\alpha^n$ converges narrowly to $\alpha:=\Lambda_\# \gamma$. Since $\PP_{\frac{1}{N}}(X)$ is closed (see Lemma \ref{L:closed}), it follows from the Portmanteau theorem (see, e.g., \cite{Bi99}) and the fact that $\alpha^n(\PP_{\frac{1}{N}}(X))=1$ for every $n$ that  $\alpha(\PP_{\frac{1}{N}}(X))=1$. (Here we use the following part of the Portmanteau theorem: if $C\subseteq \PP(X)$ is closed and $\alpha^n\in\PP(\PP(X))$ converges narrowly to $\alpha$, then $\alpha(C)\ge \limsup_{n\to\infty}\alpha^n(C)$.)  Finally,  thanks to the narrow continuity of $F_{N,k}$ we deduce the desired representation by integrating \eqref{approxrep} against a test function and passing to the limit $n\to \infty$:
 \[\mu_k=\int_{\PP_{\frac{1}{N}}(X)} F_{N,k}(\lambda) \mbox{d} \alpha(\lambda).\]
 
 \smallskip
 
 Finally, assume $k=N$ and that $\mu\in \PPs(X^N)$ can be written as 
 \begin{equation}\label{choquetNNN}
 \mu=\int_{\PP_{\frac{1}{N}}(X)} F_{N,N}(\lambda) \mbox{d} \alpha(\lambda).
 \end{equation}
 Let $\psi$ : $\PP(X)\to \R$ be bounded and continuous for the narrow topology so that $\psi \circ \Lambda \in C_b(X^N)$. Let us now observe that if $\lambda=\Lambda(\xx)\in \PP_{\frac{1}{N}}(X)$ for some $\xx=(x_1, \ldots, x_N)\in X^N$, then $S_N \delta_{\xx}=F_{N,N}(\lambda)$ and since $\Lambda(\xx)=\Lambda(x_{\sigma(1)}, \ldots, x_{\sigma(N)})=\lambda$, for every $\sigma \in {\cal{S}}_N$, we have
 \[\int_{X^N} \psi\circ \Lambda \; \mbox{d} F_{N,N}(\lambda)=\psi(\lambda). \]
Taking $\psi\circ \Lambda$ as a test function in \pref{choquetNNN} (recall that $\Lambda$ and hence $\psi\circ\Lambda$ is continuous)  yields
\[\int_{X^N} \psi \circ \Lambda \; \mbox{d} \mu= \int_{\PP_{\frac{1}{N}}(X)} \psi(\lambda) \, \mbox{d} \alpha(\lambda)\]
i.e. $\alpha=\Lambda_\# \mu$, showing in particular the uniqueness of $\alpha$.

  \end{proof}
  
 From Theorem \ref{Choquet} and the law of large numbers, we easily deduce how to recover the prior $\alpha$ from sampling, as emphasized in our introduction.

 \begin{coro}\label{C:sampling}
Let $Z=(Z_1, \ldots, Z_N)$ be a finitely exchangeable  sequence of random variables with values in $X$, let $\mu\in \PPs(X^N)$ be the law of $Z$ and let $\alpha \in \PP(\PP(X))$ be such that $\alpha(\PP_{\frac{1}{N}}(X))=1$ and 
  \begin{equation}\label{choquetformulN}
  \mu=\int_{\PP_{\frac{1}{N}}(X)} F_{N,N}(\lambda) \, \mbox{d} \alpha(\lambda).
  \end{equation}
Let $(Z^{(\nu)})_{\nu\in \N}$ be i.i.d drawn according to $\mu$, and consider the $\PP(X)$-valued sequence
\[\Lambda(Z^{(\nu)}):=\frac{1}{N} \sum_{i=1}^N \delta_{Z_i^{(\nu)}}. \;  \]
Then, almost surely, the empirical measure  $\frac{1}{n} \sum_{\nu=1}^n \delta_{\Lambda(Z^{(\nu)})}$
converges narrowly to $\alpha$ as $n\to \infty$.
 \end{coro}

\begin{proof}
We have seen in the proof of Theorem \ref{Choquet} that $\alpha=\Lambda_{\#} \mu$ so that $\Lambda(Z)$ has law $\alpha$. Hence $(\Lambda(Z^{(\nu)})_\nu$ are i.i.d $\PP(X)$-valued drawn according to $\alpha$. Since $\PP(X)$ endowed with the topology of narrow convergence is a Polish space, the claim follows from the strong law of large numbers for empirical measures on Polish spaces. 
\end{proof}

\subsection{Independent sequences as a convex mixture of extremal exchangeable sequences}

Let us compare, in a simple example, the traditional ``basis'' for representing exchangeable sequences, independent sequences, with the new one advocated in this paper, extremal exchangeable sequences. This example illustrates that  extremal exchangeables may not be a convex mixture of independents, but that independents are always a unique convex mixture of extremal exchangeables (Theorem \ref{Choquet}).  
\\[1mm]
\noindent
{\bf Example.} Let $X$ be the finite state space consisting of the three colors red, green, and blue, that is to say
$$
   X = \{ \ttr, \, \ttg, \, \ttb\},
$$
and let $N=k=3$. Consider the probability measure 
$$
    \lambda_*=\tfrac23\delta_{{\tt r}} + \tfrac13 \delta_{{\tt g}},
$$
which corresponds to an {\tt rrg} urn (i.e. an urn containing two red balls and one green ball). Now consider a sequence of three independent random variables with law $\lambda_*$, that is to say 
$$
      (Z_1,Z_2,Z_3) \sim  \lambda_*^{\otimes 3}.
$$
{\it This sequence corresponds to three independent draws with replacement from an {\tt rrg} urn}. Our goal is to find the unique representation of this joint law as a convex mixture of extreme exchangeables. 

Denote the above $\lambda_*$ by ${\tt rrg}$, and similiarly $\delta_{\ttr}$ by ${\tt rrr}$ etc. By Theorem \ref{Choquet} and the fact that $\lambda_*$ contains only red and green balls, the joint law $\lambda_*^{\otimes 3}$ must be representable as a convex combination of those $F_{N,3}(\lambda)$ where $\lambda$ contains only red and green balls, that is, $\lambda\!\! =${\tt rrr}, {\tt rrg}, {\tt rgg}, {\tt ggg}. To determine the convex combination we need to compute the $F_{N,3}(\lambda)$, which is a straightforward task given our explicit formula from section \ref{sec:univpoly},
$$
  F_{N,3}(\lambda) = \frac{N^2}{(N-1)(N-2)} \left[\lambda^{\otimes 3} -\frac{3}{N} S_3\Big((\id^{\otimes 2}_\# \lambda) \otimes \lambda\Big)+\frac{2}{N^2} \id^{\otimes 3}_\# \lambda\right].
$$  
The result is given in the following table, where we identify probability measures $\lambda = \lambda_1\delta_{\ttr}+\lambda_2\delta_{\ttg}+\lambda_3\delta_{\ttb}$ in $\PP(X)$ with their coefficient vectors in $\R^3$, and -- analogously -- probability measures in $\PP(X^3)$ with their coefficient tensors  in $\R^{3\times 3\times 3}$.
\begin{center}
\begin{tabular}{| l | p{20mm} p{20mm} p{20mm} |} 
\hline
                  & $\bigl(F_{N,3}(\lambda)\bigr)_{\cdot \, \cdot \, 1}$ 
                  & $\bigl(F_{N,3}(\lambda)\bigr)_{\cdot \, \cdot \, 2}$
                  & $\bigl(F_{N,3}(\lambda)\bigr)_{\cdot \, \cdot \, 3}$ \\[1mm] 
\hline
$\lambda=\ttr\ttr\ttr$ & 
       $\begin{pmatrix} 1 &   & \\
                          & 0 & \\
                          &   & 0 \end{pmatrix}$ & 
        $\begin{pmatrix} 0 &   & \\
                          & 0 & \\
                          &   & 0 \end{pmatrix}$ & 
        $\begin{pmatrix} 0 &   & \\
                          & 0 & \\
                          &   & 0 \end{pmatrix}$                    \\[3mm]
\hline
$\lambda=\ttr\ttr\ttg$ & 
       $\begin{pmatrix} 0 & \tfrac13   & \\
                        \tfrac13  & 0 & \\
                          &   & 0 \end{pmatrix}$ & 
        $\begin{pmatrix} \tfrac13 &   & \\
                          & 0 & \\
                          &   & 0 \end{pmatrix}$ & 
        $\begin{pmatrix} 0 &   & \\
                          & 0 & \\
                          &   & 0 \end{pmatrix}$                    \\[3mm]
\hline
$\lambda=\ttr\ttg\ttg$ & 
       $\begin{pmatrix} 0 &       & \\
                          & \tfrac13  & \\
                          &   & 0 \end{pmatrix}$ & 
        $\begin{pmatrix} 0 & \tfrac13  & \\
                         \tfrac13 & 0 & \\
                          &   & 0 \end{pmatrix}$ & 
        $\begin{pmatrix} 0 &   & \\
                          & 0 & \\
                          &   & 0 \end{pmatrix}$                    \\[3mm]
\hline
$\lambda=\ttg\ttg\ttg$ & 
       $\begin{pmatrix} 0 &   & \\
                          & 0 & \\
                          &   & 0 \end{pmatrix}$ & 
        $\begin{pmatrix} 0 &   & \\
                          & 1 & \\
                          &   & 0 \end{pmatrix}$ & 
        $\begin{pmatrix} 0 &   & \\
                          & 0 & \\
                          &   & 0 \end{pmatrix}$                    \\[3mm]
\hline                         
\end{tabular}
\end{center}
On the other hand, the joint measure we seek to represent is
$$
   \bigl(\lambda_*^{\otimes 3}\bigr)_{\cdot \, \cdot \, 1}\! = \!
   \begin{pmatrix}
       \tfrac{8}{27} & \tfrac{4}{27} & \\[0.5mm]
       \tfrac{4}{27} & \tfrac{2}{27} & \\[0.5mm]
                     &               & 0
   \end{pmatrix}\! , \;\;
   \bigl(\lambda_*^{\otimes 3}\bigr)_{\cdot \, \cdot \, 2}\! = \! \begin{pmatrix}
       \tfrac{4}{27} & \tfrac{2}{27} & \\[0.5mm]
       \tfrac{2}{27} & \tfrac{1}{27} & \\[0.5mm]
                     &               & 0
   \end{pmatrix}\! , \;\;
   \bigl(\lambda_*^{\otimes 3}\bigr)_{\cdot \, \cdot \, 3}\! = \! 
   \begin{pmatrix}
        0     &    & \\[0.5mm]
              &  0 & \\[0.5mm]
              &    & 0
   \end{pmatrix}\! .
$$
From the above explicit expressions one sees that 
\begin{equation}\label{urnex}
   \lambda_*^{\otimes 3} = \tfrac{8}{27} F_{N,3}(\ttr\ttr\ttr) +\tfrac49 F_{N,3}(\ttr\ttr\ttg) + \tfrac29 F_{N,3}(\ttr\ttg\ttg)+\tfrac{1}{27} F_{N,3}(\ttg\ttg\ttg).
\end{equation}
Thus the probability measure $\alpha$ in $\PP(\PP(X))$ in \eqref{choquetformula} is in our case given by 
$
   \alpha = \tfrac{8}{27}\delta_{\ttr\ttr\ttr} +\tfrac49 \delta_{\ttr\ttr\ttg} + \tfrac29 \delta_{\ttr\ttg\ttg}+\tfrac{1}{27}\delta_{\ttg\ttg\ttg}.
$

The representation \eqref{urnex} has the following 
probabilistic meaning. 
{\it One can simulate drawing three balls with replacement from an \ttr\ttr\ttg$\,$ urn by 
\begin{itemize}
    \item 
first picking one of the four urns \ttr\ttr\ttr, \ttr\ttr\ttg, \ttr\ttg\ttg, \ttg\ttg\ttg $\,$ with probabilities $\tfrac{8}{27}$, $\tfrac{4}{9}$, $\tfrac29$, $\tfrac{1}{27}$ (i.e. probability ratios $8 \, :\, 12\, :\, 6\, : \, 1$), 
   \item
   then drawing three balls without replacement from the picked urn.
\end{itemize}
} 
\noindent
Moreover thanks to the uniqueness result in Theorem \ref{Choquet}, the above choice of urns and probabilities provides the unique way of simulating the given draws with replacement by draws {\it without} replacement. 

We emphasize that the converse (simulating draws without replacement by draws {\it with} replacement) is not possible, due to the well known fact that finite exchangeable laws may not be representable as convex mixtures of independents. For instance, $F_{N,3}(\ttr\ttr\ttg)$ (the joint law for three draws without replacement from an \ttr\ttr\ttg$\,$ urn), being extremal thanks to Theorem \ref{T:expan}, is not a convex combination of any other joint laws, and in particular not of any independent joint laws.
\\[2mm]
Let us also provide a simple example which shows that uniqueness of the measure $\alpha$ in \eqref{choquetformula} can fail when $k<N$. 
\\[2mm]
{\bf Example.} Let $X$, $N$, and $\lambda_*$ be as above, but $k=2$. The measure $\lambda_*^{\otimes 2}$ (corresponding to two independent draws from a $\ttr\ttr\ttg\,$ urn) cannot just be  represented by the right hand side of \eqref{urnex} with $F_{N,3}$ replaced by $F_{N,2}$, but alternatively by 
$\tfrac29 F_{N,2}(\ttr\ttr\ttr) + \tfrac23 F_{N,2}(\ttr\ttr\ttg) + \tfrac19 F_{N,2}(\ttg\ttg\ttg)$, as the reader can easily check. 

%
%
%
  
  \subsection{Extremal $N$-representable $k$-plans}  
  
  The integral representation given by \pref{choquetformula} in Theorem \ref{Choquet} will allow  us to identify the set of extreme points of  $\PPnr(X^k)$ as
  \begin{equation}\label{defEnk}
  \begin{split}
   E_{N,k}&:=\{F_{N,k}(\lambda) \, : \,  \lambda\in \PP_{\frac{1}{N}}(X)\}\\
   &=\{M_k S_N \delta_{x_1, \ldots, x_N} \, : \, (x_1, \ldots, x_N)\in X^N\}
   \end{split}
   \end{equation}
  and to show in addition that all these points are exposed. Recall that if $C$ is a convex subset of $\PP(X^k)$ then $\mu\in C$ is an exposed point of $C$ if there exists $\varphi \in C_b(X^k)$ such that $\int_{X^k} \varphi\, \mbox{d} \mu <\int_{X^k} \varphi \, \mbox{d} \nu$ for every $\nu \in C\setminus\{\mu\}$. It is obvious that exposed points are extreme points but the converse need not be true in general. For the  set of $N$-representable $k$-plans, we have the following

  \begin{thm}\label{extexp} Let $N\ge k\ge 2$. \\[1mm]
a) The set of extreme points of $\PPnr(X^k)$ is given by the set $E_{N,k}$ defined in \pref{defEnk}. \\[1mm]
b) Every such extreme point is also exposed. 
  \end{thm}

  \begin{proof}
 a) Let $\mu$ be an extreme point of $\PPnr(X^k)$.  Let us write $\mu$ as in \pref{choquetformula} and prove that $\alpha$ is a Dirac mass. If this was not the case, we could find $\varphi\in C_b(X)$ and $t\in \R$ such that  $A_1:=\{\lambda \in \PP(X) \; : \; \int_X \varphi \, \mbox{d} \lambda>t\}$ and $A_2:=\{\lambda \in \PP(X) \; : \; \int_X \varphi \, \mbox{d} \lambda \leq t\}$ satisfy $\alpha(A_1)>0$ and $\alpha(A_2)>0$. Then decomposing $\alpha$ as $\alpha(A_1)\alpha_1 + \alpha(A_2)\alpha_2$ with $\alpha_i$ a probability measure giving full mass to $\PP_{\frac{1}{N}}(X)\cap A_i$, the extremality of $\mu$ would imply
 \[\int_{\PP_{\frac{1}{N}}(X)} F_{N, k}(\lambda)\, \mbox{d} \alpha_1(\lambda)=\int_{\PP_{\frac{1}{N}}(X)}  F_{N, k}(\lambda) \, \mbox{d} \alpha_2(\lambda). \]
 But recalling that for $\lambda \in \PP_{\frac{1}{N}}(X)$ we have $M_1 F_{N,k}(\lambda)=\lambda$, this would also give
 \[\int_{\PP_{\frac{1}{N}}(X)} \lambda \, \mbox{d} \alpha_1(\lambda)=\int_{\PP_{\frac{1}{N}}(X)}  \lambda \, \mbox{d} \alpha_2(\lambda), \]
 and integrating $\varphi$ against this identity would lead to a contradiction. Extreme points of $\PPnr(X^k)$ therefore belong to $E_{N,k}$

The reverse inclusion follows from b). Alternatively, it follows from Theorem \ref{theo-extremeabstract}, for if $\mu=F_{N,k}(\lambda)$ for some $\lambda\in\PP_{\frac{1}{N}}(X)$, and $\mu$ is a strict convex combination of two measures $\mu^1\neq \mu^2$ in $\PPnr(X^k)$, then the support of the $\mu^i$ must be contained the support of $\mu$ and thus -- a fortiori -- in $(\rm{supp}\, \lambda)^k$, contradicting the extremality of $F_{N,k}(\lambda)$ among $N$-representable $k$-point measures supported on the finite state space $({\rm supp}\, \lambda)^k$. 

b) We need to show that if $\lambda:=\frac{1}{N} \sum_{i=1}^N \delta_{x_i}\in \PP_{\frac{1}{N}}(X)$, then $\mu:=F_{N,k}(\lambda)$ is an exposed point of $\PPnr(X^k)$. Again thanks to the integral representation \pref{choquetformula}, this amounts to finding $\varphi\in C_b(X^k)$ such that $\int_{X^k} \varphi \mbox{d} \mu <\int_{X^k} \varphi \mbox{d} \nu$ for every $\nu \in E_{N,k}\setminus\{\mu\}$.

  \smallskip
  
  First, let us rewrite $\lambda:=\sum_{j=1}^\ell \lambda_j  \delta_{y_j}$ with $\lambda_j>0$ and $y_1, \ldots, y_\ell$ pairwise distinct. Next, we define
  \[\varphi_0(z_1, \ldots, z_k):= \sum_{i=1}^k \min_{j=1, \ldots, \ell} d(z_i, y_j), \; \forall (z_1, \ldots, z_k)\in X^k.\]
 Let $u$ be a function in $C_b(X, \R^\ell)$ such that $u(y_1), \ldots, u(y_\ell)$ are linearly independent. For every $\theta\in \PP(X)$, let us now define
 \[\begin{split}
 J(\theta):&=\frac{1}{2} \int_{X\times X} \vert u(x)-u(y)\vert^2 \mbox{d} \theta(x) \mbox{d}\theta(y)\\ 
 &= \int_X \vert u(x)\vert^2 \mbox{d} \theta(x)- \Big \vert \int_X u(x) \mbox{d} \theta(x)\Big\vert^2
 \end{split}\]
 Obviously $J$ is a concave quadratic functional and more precisely, defining
 \begin{equation}\label{defflamb}
 f_{\lambda}(x):=\vert u(x)\vert^2-  2 \Big(\int_X  u(y) \mbox{d} \lambda(y)\Big) \cdot u(x), \; \forall x\in X,
 \end{equation}
 for every $\theta\in \PP(X)$ we have  
  \[J(\theta)=J(\lambda)+\int_X f_{\lambda}(x) \mbox{d}(\theta-\lambda)(x)-\Big \vert   \int_X  u(x) \mbox{d} (\theta-\lambda)(x) \Big\vert^2.\]
  In particular, this implies that
  \begin{equation}\label{concavineqJ}
  J(\theta)\le J(\lambda)+\int_X f_{\lambda}(x) \mbox{d}(\theta-\lambda)(x)
  \end{equation}
  and this inequality is strict unless 
 \begin{equation}\label{equalitycase}
  \int_X u \mbox{d} \theta=\int_X u \mbox{d} \lambda=\sum_{j=1}^l  \lambda_j u(y_j).
  \end{equation}

  Now if $\theta\in \PP_{\frac{1}{N}}(X)$ and $\nu:=F_{N,k}(\theta)$, we know (see \pref{mu2}) that
  \[M_2 \nu=\frac{N}{N-1} \theta \otimes \theta -\frac{1}{N-1} \id^{\otimes 2}_\# \theta, \; M_1 \nu=\theta\]
  since $(x,y)\in X^2 \mapsto \vert u(x)-u(y)\vert^2$ vanishes on the diagonal and thanks to \pref{concavineqJ}, we thus get
  \[\begin{split}
\frac{N}{N-1} J(\theta)&=  \frac{1}{2} \int_{X^k} \vert u(z_1)-u(z_2)\vert^2 \mbox{d} \nu(z_1, \ldots, z_k) \\
&\leq     \frac{1}{2} \int_{X^k} \vert u(z_1)-u(z_2)\vert^2 \mbox{d} \mu(z_1, \ldots, z_k)\\
& +\frac{N}{N-1} \int_{X^k} f_{\lambda}(z_1) \mbox{d} (\nu-\mu)(z_1, \ldots, z_k)
  \end{split}\]  
 and the last inequality is strict unless \pref{equalitycase} holds. Defining
 \[\varphi_1(z_1, \ldots, z_k):=-\frac{1}{2} \vert u(z_1)-u(z_2)\vert^2 +\frac{N}{N-1} f_{\lambda}(z_1), \;   \forall (z_1, \ldots, z_k)\in X^k\]
 we deduce that $\int_{X^k} \varphi_1 \mbox{d} \mu \le \int_{X^k} \varphi_1 \mbox{d} \nu, \; \forall \nu \in E_{N,k}$. Since $\int_{X^k} \varphi_0 \mbox{d} \mu=0$ and $\varphi_0\geq 0$, setting $\varphi:=\varphi_0+\varphi_1$ we also have
 \begin{equation}\label{minimu}
 \int_{X^k} \varphi \mbox{d} \mu \le \int_{X^k} \varphi \mbox{d} \nu, \; \forall \nu \in E_{N,k}.
 \end{equation}
 It remains to show that if  $\theta \in \PP_{\frac{1}{N}}(X)$ is such that   \pref{minimu} is an equality at $\nu=F_{N,k}(\theta)$ then necessarily $\theta=\lambda$. But in this equality case we must have $\int_{X^k} \varphi_0 \mbox{d} \nu=0$ so that $\theta= M_1 \nu$ is supported by the finite set $\{y_1, \ldots y_\ell\}$ (i.e. $\theta=\sum_{j=1}^\ell \theta_j \delta_{y_j}$), but we must also have an equality in \pref{concavineqJ}. The latter implies that $\sum_{j=1}^\ell \theta_j u(y_j)=\sum_{j=1}^\ell \lambda_j u(y_j)$, i.e. $\lambda_j=\theta_j$ for  $j=1, \ldots, \ell$, since we have chosen $u(y_1), \ldots, u(y_\ell)$ linearly independent. 
\end{proof}

Let us point out the following consequence of Theorem \ref{extexp} a). 

\begin{coro} The set of extreme points of $\PPnr(X^k)$ is closed under narrow convergence. 
\end{coro}

\begin{proof} We have to show that $E_{N,k}$ is closed under narrow convergence. Take a sequence $\mu^n = F_{N,k}(\lambda^n)\in E_{N,k}$ converging narrowly to some $\mu\in\PP(X^k)$. Since the one-point marginals $M_1\mu^n = \lambda^n$ belong to $\PP_{\frac{1}{N}}(X)$ and the marginal map $M_1$ is continuous under narrow convergence, $\lambda^n\to\lambda$ for some $\lambda\in\PP(X)$. By the closedness of $\PP_{\frac{1}{N}}(X)$ (see Lemma \ref{L:closed}) we have $\lambda\in\PP_{\frac{1}{N}}(X)$ and the continuity of $F_{N,k}$ yields $F_{N,k}(\lambda^n)\to F_{N,k}(\lambda)$, so $\mu=F_{N,k}(\lambda)\in E_{N,k}$. 
\end{proof}

In some applications, one is interested in the following subset of $\calP_{N-rep}(X^k)$ consisting of measures with no mass on the diagonal:
\begin{equation} \label{offdiag}
   \calP_{N-rep}^{\rm offdiag}(X^k) := \{ \mu\in\calP_{N-rep}(X^k) \, : \, \mu(diag_k(X))=0 \}
\end{equation}
where
\begin{equation} \label{diag}
   diag_k(X) := \{ (x,...,x)\in X^k \, : \, x\in X \}.
\end{equation}  
The motivation for considering this set comes from singular $k$-body interactions. These interactions lead to costs of the form $c(x_1,...,x_N)=\sum_{1\le i_1<...<i_k\le N}\Phi(x_{i_1},...,x_{i_k})$ for some measurable nonnegative $\Phi$ with $\Phi=+\infty$ on $diag_k(X)$, then the total cost 
$C[\gamma]=\int_{X^N}c\, d\gamma$ is infinite if $M_k\gamma\not\in \calP_{N-rep}^{\rm offdiag}(X^k)$. A prototypical example is the Coulomb cost, see Section \ref{sec-opti}, which led Khoo and Ying \cite{KhooYing} to introduce the set \eqref{offdiag}--\eqref{diag} in the case $k=2$ and $X$ a finite state space. Recall the quantization constraint for measures $\lambda\in\calP_{\frac{1}{N}}(X)$ that $\lambda(\{x\})\in\{0,1/N,...,1\}$. 
\begin{thm} \label{extexpoff} a) The set of extreme points of $\calP_{N-rep}^{\rm offdiag}(X^k)$ is given by 
\begin{equation} \label{ENkoffdiag}
   E_{N,k}^{\rm offdiag} := \Bigl\{ F_{N,k}(\lambda )\, : \, \lambda\in\calP_{\frac{1}{N}}(X) \, : \, 
       \lambda( \{ x\} )\in \{ 0,\tfrac{1}{N},...,\tfrac{k-1}{N} \} \Bigr\}. 
\end{equation}
b) Every such extreme point is also exposed. 
\end{thm}
Geometrically this means that intersecting the set of $N$-representable $k$-plans with the subspace of measures which vanish on the diagonal does not generate new extreme points.

Theorem \ref{extexpoff} generalizes a recent result of \cite{KhooYing} from $k=2$ and finite state spaces $X$ to general $k$ and $X$. 
\\[2mm]
\begin{proof}
First we show that any extreme point $\mu$ is contained in the set \eqref{ENkoffdiag}. By the representation \eqref{choquetformula}, $\mu=\int_A F_{N,k}(\lambda) \, d\alpha(\lambda)$ where $A=\{ \lambda\in\calP_{\frac{1}{N}}(X)\, : \, F_{N,k}(\lambda)(diag_k(X))=0\}$. The key point is that the polynomial $F_{N,k}(\lambda)$ factorizes on the diagonal: 
\begin{eqnarray} 
   F_{N,k}(\lambda ) (\{ (x,...,x)\} ) &=& \frac{N^{k-1}}{\prod_{j=1}^{k-1}(N-j)}\Bigl[ \lambda( \{ x\} )^k 
   + \sum_{j=1}^{k-1}\frac{(-1)^j}{N^j} \lambda( \{ x\} )^{k-j} c_j^{(k)} \Bigr] \nonumber \\
    & = &  \frac{N^{k-1}}{\prod_{j=1}^{k-1}(N-j)} \lambda( \{ x \} )\cdot \Bigl( \lambda( \{ x\} )-\tfrac{1}{N}\Bigr) \cdot ... \cdot \Bigl( \lambda( \{ x\} )-\tfrac{k-1}{N}\Bigr). \label{factorize}
\end{eqnarray}
The second equality follows because the coefficients $c_j^{(k)}$ are $N$-independent and satisfy 
$$
  N^k \Bigl( 1 - \frac{c_1^{(k)}}{N} + \frac{c_2^{(k)}}{N^2} -+... + (-1)^{k-1}
  \frac{c_{k-1}^{(k)}}{N^{k-1}}\Bigr) = \prod_{j=0}^{k-1} (N-j)
$$
for all nonnegative integers $N$, see eq.~\eqref{ceq}. Since $\lambda$ is $\tfrac{1}{N}$-quantized, $\lambda({x})N$ is a nonnegative integer and can be substituted for $N$, yielding the asserted expression. The factorization identity \eqref{factorize} shows that $A=\{\lambda\in\calP_{\frac{1}{N}}(X) \, : \, \lambda({x}) \in\{0,\tfrac{1}{N},...,\tfrac{k-1}{N}\}\}$.
The rest of the proof of a) is analogous to that of Theorem \ref{extexp} a) and the exposedness of the elements in $E_{N,k}^{\rm offdiag}$ is immediate from Theorem \ref{extexp} b).

\end{proof}
\subsection{Exponentially convergent approximation of extreme points}

The error estimates \eqref{err}, \eqref{errp} when approximating $F_{N,k}(\lambda)$ by a truncated series can be extended in a straightforward manner to Polish spaces.   
  
\begin{thm} \label{T:expon} For any extreme point $\mu_k$ of $\PPnr(X^k)$, denoting its one-point marginal $M_1\mu_k\in\PP_{\frac{1}{N}}(X)$ by $\lambda$ we have
\begin{equation} \label{errtilde}
   \mu_k = \frac{N^{k-1}}{ \prod_{j=1}^{k-1}(N-j)} \left[ \lambda^{\otimes k} + \eps_{N,k}(\lambda)\right] \mbox{ with }  \Vert \eps_{N, k} (\lambda)\Vert_{\TV} \le \frac{C_k}{N}
\end{equation}
and 
\begin{align} 
   \mu_k = & \frac{N^{k-1}}{ \prod_{j=1}^{k-1}(N-j)} \left[ \lambda^{\otimes k} + \sum_{j=1}^p  \frac{(-1)^j}{N^j} S_k \, P_j^{(k)}(\lambda ) + 
\eps_{N,k,p}(\lambda)\right] \nonumber \\ & \; \mbox{ with }  \Vert \eps_{N, k,p} (\lambda)\Vert_{\TV} \le \frac{C_k}{N^{p+1}}, \label{errptilde}
\end{align}
with constants $C_k$ independent of $N$ and $p$. The explicit constants \eqref{Ck} will do. Conversely, for every $\lambda\in\PP_{\frac{1}{N}}(X)$ there exists an extreme point $\mu_k$ of $\PPnr(X^k)$ such that \eqref{errtilde}, \eqref{errptilde} hold.  
\end{thm}

\begin{proof} By Theorem \ref{extexp}, an extreme point $\mu_k$ equals $F_{N,k}(\lambda)$, and in particular is supported on the finite set $({\rm supp}\,\lambda)^k$. Estimates \eqref{errtilde}, \eqref{errptilde} now follow from \eqref{err}--\eqref{errp} and the fact that
$$
   ||\eps_{N,k}(\lambda)||_{TV(X^k)} = ||\eps_{N,k}(\lambda)||_{TV((\mbox{\scriptsize supp}\,\lambda)^k)}
$$  
(and the analogous identity for $\eps_{N,k,p}$).
\end{proof}  

Of course, via the integral representation \eqref{choquetformula} this translates into an analogous exponentially convergent approximation result for arbitrary (non-extremal) $N$-representable measures, see the next section. 
  
\section{Recovering the de Finetti-Hewitt-Savage theorem}\label{sec-hs}


The celebrated Hewitt-Savage theorem \cite{Hewitt-Savage} (going back to de Finetti in the special case of the state space $X=\{0,1\}$) says the following.

\begin{thm} \label{T:HS} {\rm (Hewitt-Savage \cite{Hewitt-Savage})} If $(Z_n)_n$ is an infinite exchangeable  sequence of random variables with values in the Polish space $X$, then there exists $\alpha \in \PP(\PP(X))$ such that for every $k$, the law $\mu_k$ of $(Z_1, \ldots, Z_k)$ is given by
\begin{equation}\label{hsrep}
\mu_k =\int_{\PP(X)} \lambda^{\otimes k} \; \mbox{ d} \alpha(\lambda).
\end{equation}
\end{thm}

Let us show how this theorem follows from Theorem \ref{Choquet} which therefore may be viewed as a non-asymptotic form of the Hewitt-Savage theorem. If $\mu_k$ is the law of the first $k$  components of an infinite exchangeable  sequence, it is $N$-representable for every $N$, so that there exists $\alpha^{N,k} \in \PP(\PP_{\frac{1}{N}}(X))$ such that
\[ 
     \mu_k=\int_{\PP_{\frac{1}{N}}(X)} F_{N,k}(\lambda) \mbox{d} \alpha^{N,k}(\lambda).
\]
In particular, for the law $\mu_1$ of $Z_1$ we have
\[
   \mu_1:=\int_{\PP(X)} \lambda \; \mbox{d} \alpha^{N,k}(\lambda).
\]
Since $X$ is Polish, $\mu_1$ is tight so there exists $\Psi$ : $X\to [0, +\infty]$ which is lower semicontinuous and coercive (i.e. with compact sublevel sets) such that
\[
   +\infty> \int_X \Psi(x) \mbox{d} \mu_1(x)=\int_{\PP(X)} \Big (\int_{X} \Psi(x) \mbox{d} \lambda(x) \Big) \mbox{d} \alpha^{N,k}(\lambda).
\]
But since $\Psi$ is coercive, $\lambda \in \PP(X) \mapsto \int_{X} \Psi(x) \mbox{d} \lambda(x)$ is coercive on $\PP(X)$. It thus follows from Prokhorov's theorem that $\alpha^{N,k}$ is tight. Hence, taking a subsequence if necessary, $\alpha^{N,k}$ converges narrowly  to  some $\alpha^k$ as $N \to \infty$, and since $F_{N,k}(\lambda)$ converges uniformly to $\lambda^{\otimes k}$, thanks to \eqref{err} we get letting $N \to \infty$:
\[ \mu_k = \int_{\PP(X)} \lambda^{\otimes k} \; \mbox{d} \alpha^k(\lambda).\]

It remains to show that $\alpha^k$ may be chosen independently of $k$. For this we simply observe that for $\ell\ge k$, $\mu_k =M_k(\mu_\ell)$. Since $M_k(\lambda^{\otimes \ell} )=\lambda^{\otimes k}$, thanks to the linearity and narrow continuity  of $M_k$ we thus get 
\[\mu_k=\int_{\PP(X)} \lambda^{\otimes k} \; \mbox{d} \alpha^\ell(\lambda), \; \mbox{ for all } \ell\ge k.\]
We then use the same tightness argument as above to see that for some sequence $\ell_n \to \infty$, $\alpha^{\ell_n}$ converges narrowly to some  $\alpha$ as $n\to +\infty$ and the Hewitt-Savage representation \pref{hsrep} holds for this measure $\alpha$.


\smallskip

Diaconis and Freedman (in \cite{Dia-Free}) obtained ${\rm{TV}}$ bounds between elements of $\PPnr(X^k)$ and mixtures of independent measures $\lambda^{\otimes k}$. More precisely, Diaconis and Freedman showed that if $\mu_k \in \PPnr(X^k)$ and $k\leq N$, there exists $\alpha \in \PP(\PP(X))$ such that, for some constant $B_k$ (for  results on their $k$-dependence see \cite{Dia-Free})
\begin{equation} \label{DiaFre}
   \Vert \mu_k -\int_{\PP(X)} \lambda^{\otimes k} \mbox{d} \alpha\Vert_{\TV} \leq \frac{B_k}{N}.
\end{equation}
These $O(\frac{1}{N})$ bounds are recovered 
by using the integral representation of $\mu_k \in \PPnr(X^k)$ of Theorem \ref{Choquet} and applying \pref{errtilde}. 

A better approximation of $\mu_k$ is obtained by adding the universally correlated correction terms from \eqref{errptilde} to the independent measures $\lambda^{\otimes k}$ in \eqref{DiaFre} and applying \eqref{errptilde}. This yields:
\begin{coro} For any measure $\mu_k$ in $\PPnr(X^k)$, and with $\alpha$ as in Theorem \ref{choquetformula}, we have for $p=1,...,k-2$ and some constant $B_k$
\begin{equation} \label{DiaFreImproved}
    \Vert \mu_k -\int_{\PP(X)} \tfrac{N^{k-1}}{\prod_{j=1}^{k-1}(N-j)}\bigl[\lambda^{\otimes k}  + \sum_{j=1}^p \frac{(-1)^j}{N^j} S_kP_j^{(k)}(\lambda )\bigr] \, \mbox{d} \alpha\Vert_{\TV} \leq \frac{B_k}{N^{p+1}}\; .
\end{equation}
\end{coro}



\section{Connection with the Ewens sampling formula}\label{sec-ewens}

When looking at the coefficients arising in the classification \eqref{expan}--\eqref{coeffP}  
of extremal $N$-representable measures, we noticed some resemblance to those in the celebrated {\it Ewens distribution} or {\it Ewens sampling formula} from genetics, which -- mathematically speaking -- is a probability distribution on integer partitions. On the other hand, the sign factor $(-1)^j$ in \eqref{expan}, or $(-1)^{k-n(\bfp')}$ in the alternative expression \eqref{nonrecurmu2}, means that some coefficients are negative, so the coefficients cannot be a probability distribution on integer partitions. {\it What is happening here?} 

There is a precise connection which seems to us quite remarkable and which we now describe. 

The Ewens distribution is usually stated in terms of the following parametrization of partitions. If $\bfp'=(p'_1,...,p'_m)$ is a partition of the integer $k$, that is to say $p'_1+...+p'_m=k$ and $p'_1\ge ... \ge p'_m\ge 1$, one denotes by $m_q$ the number of components equal to $q$ (i.e., the number of times the term $q$ appears in the sum), $m_q:=| {\bfp'} {^{-1}}(q)|$, and introduces the vector $(m_1,...,m_k)$ (called allelic partition). For instance, the partition $\bfp'=(2,1,1,1)$ of $5$ corresponds to the allelic partition $(3,1,0,0,0)$, because $1$ appears $3$ times, $2$ appears once, and $3$, $4$, $5$ appear zero times. The Ewens distribution on the allelic partitions of the integer $k\in\N$ is now given by
\begin{equation}\label{ewens}
  {\mathbb P}(m_1,...,m_k;\theta ) = \frac{k!}{\theta(\theta+1)...(\theta+k-1)} \prod_{q=1}^k \frac{\theta^{m_q}}{q^{m_q} m_q!},
\end{equation}
where $\theta>0$ is a parameter. In genetics, this distribution models the genetic diversity of a population (at statistical equilibrium and under neutral selection), and $\theta$ is the {\it mutation rate}. More precisely, for a random sample of $k$ genes taken from a population at a particular locus, eq. \eqref{ewens} gives the probability that $m_1$ alleles (variant forms of the gene) appear exactly once, $m_2$ alleles appear exactly twice, and so on. The Ewens distribution  has found widespread use in biology, statistics, probability, and even   
number theory \cite{DG93}, see \cite{Ew72, KM72} for original papers and \cite{Cr16} for a recent review.  

In terms of the standard representation of the partitions of $k$, expression \eqref{ewens} becomes
\begin{equation} \label{ewens2}
   \frac{k!}{\theta(\theta+1)...(\theta+k-1)} \, \theta^{n(\bfp' )}  \prod_{i=1}^{n(\bfp' )} \frac{1}{p'_i}  \, 
 \frac{1}{\prod\limits_{q\in \mbox{\scriptsize Ran}\, \bfp'} (| \bfp' {^{-1}}(q)|)!} \, =: {\tt ew}(\bfp';\theta ).  
\end{equation}

\begin{defi} Given any partition $\bfp'$ of the integer $k$, we define the complex Ewens function $\theta\mapsto{\tt ew}(\bfp';\theta)$ by \eqref{ewens2}, where $\theta\in\C$ is a complex parameter and the arithmetic operations are the usual ones in $\C$. 
\end{defi}

For a picture when $k=3$ see Figure \ref{F:ewens}. 

\begin{figure} 
\begin{center}
\includegraphics[width=\textwidth ]{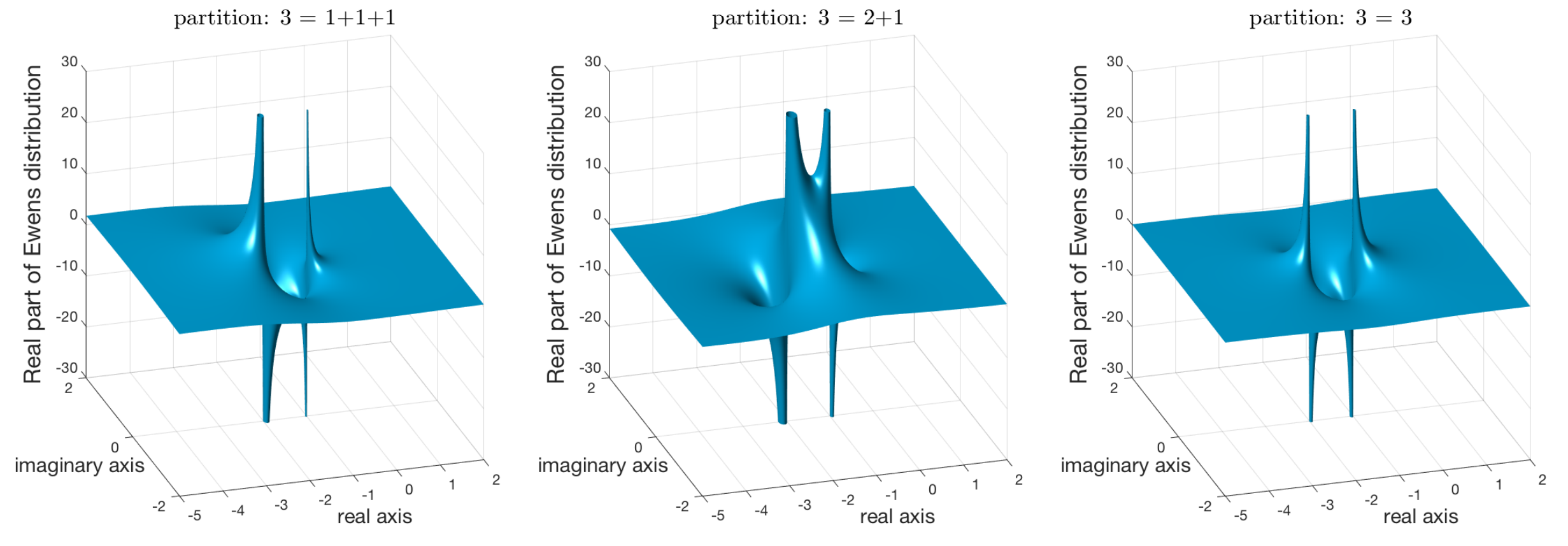}
\caption{The complex Ewens distribution $(\bfp',\theta)\mapsto {\tt ew}(\bfp';\theta)$ for $k=3$.  Since there are three partitions $\bfp'$ of $3$, it is a collection of 
three meromorphic functions of the complex variable $\theta$. The poles of each function are located at $\theta=-1$ and $\theta=-2$. 
Restricting $\theta$ to the positive real axis yields the classical Ewens distribution, which is a probability measure for each $\theta$, and has long been known to be relevant in genetics. 
Its analytic continuation to the negative real axis, which to our knowledge has not been introduced before, is a signed measure for each $\theta$, and turns out to describe extremal $N$-representable measures.}
\label{F:ewens}
\end{center}
\end{figure}

The complex Ewens function has the following properties:
\begin{enumerate}
\item ${\tt ew}(\bfp';\cdot)$ has poles exactly at the negative integers $\theta=-1$, ..., $-(k\! -\! 1)$, and is holomorphic on $\C\backslash\{-1,...,-(k\! -\! 1)\}$. In particular, ${\tt ew}(\bfp';\cdot)$ is meromorphic on $\C$. 
\item ${\tt ew}(\bfp';\cdot)$ is nonzero for the trivial partition $\bfp'=k$, and zero only at $\theta=0$ for all other partitions of $k$.
\item $\sum\limits_{\bfp' \mbox{\scriptsize partition of }k}{\tt ew}(\bfp';\cdot)$ is identically equal to $1$. 
\item $\{ {\tt ew}(\bfp';\theta)\}_{\bfp' \mbox{\scriptsize partition of }k}$ is a probability measure (i.e., in addition real and nonnegative) when $\theta$ belongs to the nonnegative  real axis, and a signed measure (i.e., in addition real but not nonnegative) when $\theta$ belongs to the negative real axis excluding the poles. 
\end{enumerate}

The first two properties are immediate from the definition. The third follows from the fact that the sum is identically equal to $1$ for positive real $\theta$ and the identity theorem from complex analysis. Property 4 is an obvious consequence of Property 3 and the definition.
 
The connection to extremal $N$-representable measures is the following. 

\begin{thm} \label{T:ewens} The coefficients $c_{\bfp'}$ in the expansion \eqref{nonrecurmu2} of extremal $N$-representable $k$-point measures are given by 
${\tt ew}(\bfp';\theta)|_{\theta=-N}$; that is, they are given by the Ewens distribution evaluated at a probabilistically meaningless but by analytic continuation admissible negative parameter value. 
\end{thm}


Properties 3 and 4 of the complex Ewens distribution thus recover the phenomenon discussed at the beginning of the section that the coefficients $c_{\bfp'}$ in \eqref{nonrecurmu2} are only a signed measure of total mass $1$ but not a probability measure. (As already noted below Theorem \ref{T:expan}, due to this phenomenon it is highly nontrivial that the induced measures \eqref{nonrecurmu}, $\lambda\in\calP_{\frac{1}{N}}(X)$, are nevertheless probability measures.) 

Some sort of correspondence between $N$-representable $k$-point measures and the Ewens distribution at $\theta=-N$ also exists at the poles. By Property 1, the complex Ewens distribution stops making mathematical sense at the points $-N$ when $N\in\N$, $N<k$; but these are precisely the $N$'s for which the notion of $N$-representability (see Definition \eqref{DefNrep}) stops making sense as a $k$-plan cannot be $N$-representable when $N<k$. Thus the two poles seen in Figure \ref{F:ewens} at $-1$ and $-2$ reflect the fact that it makes no sense for a $3$-point probability measure to be $1$- or $2$-representable.

\section{Applications to optimal transport}  \label{sec-opti}
 
 Our interest in the structure of the set $\PPnr(X^k)$ is motivated by symmetric multi-marginal optimal transport problems. More precisely, given integers $2\leq k \leq N$ and $\Phi\in C_b(X^k)$ \emph{symmetric} (i.e., $\Phi(x_{\sigma(1)}, \ldots, x_{\sigma(k)}) = \Phi(x_1, \ldots, x_k)$ for every $(x_1, \ldots, x_k)\in X^k$ and every permutation $\sigma \in  \calS_k$), we consider the symmetric cost $c_{\Phi}$ defined on $X^N$ by
\begin{equation} \label{costexpan}
  c_{\Phi}(x_1, \cdots, x_N):= \frac{1}{{{N}\choose{k}} } \sum_{1\leq i_1 <i_2\ldots <i_k\leq N} \Phi(x_{i_1}, \ldots, x_{i_k}).
\end{equation}
 In this setting $N$ is the total number of particles in the system, $X^N$ is the space of all $N$-particle configurations, and $\Phi$ is a $k$-body interaction potential. In practice $N$ is much larger than $k$. Given $\rho\in \PP(X)$ we are interested in the multi-marginal optimal transport problem
 \begin{equation}\label{mmk}
 C_{N,k}(\rho):=\inf\Big\{ \int_{X^N} c_{\Phi}(x_1, \ldots, x_N) \mbox{d}\gamma(x_1, \ldots, x_N) \; : \;   \gamma\in \PPs(X^N), \; M_1 \gamma =\rho\Big\}.
 \end{equation}
 An important example which has received much recent interest is $k=2$, $X=\R^3$, $\Phi(x_1,x_2)=|x_1-x_2|^{-1}$ (optimal transport with Coulomb cost \cite{CFK, BDG}), in which case \eqref{mmk} arises as the strictly correlated limit of Hohenberg-Kohn density functional theory \cite{CFK}. 
 Thanks to \eqref{costexpan}, the minimization can be reformulated in terms of the $k$-marginal $\mu_k=M_k(\gamma)\in \PPnr(X^k)$:
 \begin{equation}\label{mmknr}
 C_{N,k}(\rho):=   \inf\Big\{ \int_{X^k} \Phi  \mbox{d} \mu_k\; : \;   \mu_k \in \PPnr(X^k), \; M_1(\mu_k)=\rho\Big\}.
 \end{equation}

 \subsection{A $\frac{1}{N}$-quantized polynomial convexification problem}

The de Finetti style representation from Theorem \ref{Choquet}, together with the fact that $M_1(F_{N,k}(\lambda))=\lambda$ for every $\lambda\in \PP_{\frac{1}{N}}(X)$, enables us to write $C_{N,k}(\rho)$ as follows.

\begin{thm} \label{rewrite} {\rm (Reformulation of symmetric multi-marginal optimal transport)} The functional $C_{N,k}$ satisfies
\begin{equation}\label{convquantal}
C_{N,k}(\rho)=\inf   \Big\{\int_{\PP_{\frac{1}{N}}(X)}  \Big(\int_{X^k} \Phi  \;\mbox{d} F_{N,k}(\lambda)\Big) \mbox{d} \alpha(\lambda) \; : \;  \alpha \in \PP(\PP_{\frac{1}{N}}(X)), \, \eqref{margconsal} \Big\},
\end{equation}
with prescribed marginal constraint
\begin{equation}\label{margconsal}
\int_{\PP_{\frac{1}{N}}(X)} \lambda \; \mbox{d} \alpha(\lambda)=\rho. 
\end{equation}
\end{thm}

In view of formulae \eqref{expan}--\eqref{coeffP}, one can observe that $\lambda \in \PP_{\frac{1}{N}(X)} \mapsto \int_{X^k} \Phi  \;\mbox{d} F_{N,k}(\lambda)$ is a polynomial of degree $k$ expression in the weights of the discrete measure $\lambda$, for instance
\[ \int_{X^2} \Phi  \;\mbox{d} F_{N,2}(\lambda)=\frac{N}{N-1} \int_{X^2} \Phi(x,y) \mbox{d}\lambda(x) \mbox{d} \lambda(y) -\frac{1}{N-1} \int_X \Phi(x,x) \mbox{d}\lambda(x)\]
and
\[\begin{split}
 \int_{X^3} \Phi  \;\mbox{d} F_{N,3}(\lambda)&=\frac{N^2}{(N-1)(N-2)} \int_{X^3} \Phi(x,y,z) \mbox{d}\lambda(x) \mbox{d} \lambda(y) \mbox{d}\lambda(z) \\
& -\frac{3N}{(N-1)(N-2)}  \int_{X^2} \Phi(x,x,y) \mbox{d}\lambda(x) \mbox{d} \lambda(y)\\
& +\frac{2}{(N-1)(N-2)} \int_{X} \Phi(x,x,x) \mbox{d}\lambda(x) .
 \end{split}\]
Defining the polynomial $P_{N,k}$ for every single marginal (not necessarily $\frac{1}{N}$-quantized) $\lambda\in \PP(X)$ by
\begin{equation}\label{defPNk}
P_{N,k}(\lambda):=\int_{X^k} \Phi  \;\mbox{d} F_{N,k}(\lambda), \; \mbox{ for all } \lambda\in \PP(X) ,
\end{equation}
we see that minimization problem \pref{convquantal}-\pref{margconsal} is a $\frac{1}{N}$-quantized constrained version of the convexification of the polynomial $P_{N,k}$:
\begin{equation}\label{defpnkconv}
P_{N,k}^{**}(\rho):=\inf_{\alpha \in \PP(\PP(X))} \Big\{ \int_{\PP(X)} P_{N,k}(\lambda) \mbox{d} \alpha(\lambda) \; :\; \int_{\PP(X)} \lambda \mbox{d} \alpha(\lambda)=\rho\Big\}.
\end{equation}
In particular note that 
\begin{equation}\label{lowerbcc}
C_{N,k}(\rho) \geq P_{N,k}^{**}(\rho), \; \forall \rho \in \PP(X). 
\end{equation}

To understand the advantage of the formulation \pref{convquantal}-\pref{margconsal} compared to the initial multi-marginal problem \pref{mmk}, it is worth considering in detail the finite case where $X=X_\ell$ is finite with $\ell$ elements. In this case both \pref{mmk} and \pref{convquantal}-\pref{margconsal} are linear programs which have ${N+\ell-1}\choose{\ell-1}$ variables and $\ell$ constraints. But the special structure of \pref{convquantal}-\pref{margconsal} makes it much more appealing from a computational viewpoint. Indeed, the computation of the value function $C_{N,k}$ given by \pref{convquantal}-\pref{margconsal} amounts to finding the convex envelope of the restriction of the polynomial $P_{N,k}$ to the finite set $\PP_{\frac{1}{N}}(X)$. In the simplest case where $\ell=2$, this amounts to computing the convex hull of  $N+1$ points in the plane located in the graph of $P_{N,k}$. This convex envelope can be computed exactly in linear in $N$ time thanks to the Graham scan algorithm \cite{Graham} for instance.  

  \subsection{Convergence as $N\to \infty$}

Now we consider the asymptotic behavior of $C_{N, k}$ as $N\to \infty$. As first emphasized in \cite{CFP} in the case of the Coulomb cost (or more general pairwise interactions with a potential having a positive Fourier transform),  the Hewitt-Savage theorem enables one to drastically simplify the analysis of problems like \pref{mmk} as $N\to \infty$. Since $\rho^{\otimes k}$ is admissible in the optimal transport problem \pref{mmknr} we have
\[C_{N,k}(\rho)\leq P_k(\rho):= \int_{X^k}  \Phi \mbox{d} \rho^{\otimes k}, \; \forall \rho \in \PP(X)\]
but since $C_{N,k}$ is obviously convex this also gives
\begin{equation}\label{upperbcc}
C_{N,k}(\rho) \leq P_k^{**}(\rho), \; \mbox{ for all } \rho \in \PP(X).
\end{equation}
Recalling \pref{errtilde}, we observe that for some constant $C_k$ we have for every $\lambda \in \PP(X)$
\[P_{N,k}(\lambda)\geq \frac{(N-k)! N^k}{N!} \Big(P_k(\lambda)-\frac{C_k \Vert \Phi \Vert_{\infty}}{N}.  \Big)\]
Taking convex envelopes and using \pref{lowerbcc} we thus get, for every $\rho \in \PP(X)$:
\begin{equation}\label{uniformcv}
 \frac{(N-k)! N^k}{N!} \Big(P_k^{**} (\rho)-\frac{C_k \Vert \Phi \Vert_{\infty}}{N}  \Big) \leq C_{N,k}(\rho)\leq P_k^{**}(\rho).
\end{equation}
So $C_{N,k}$ converges uniformly on $\PP(X)$ to $P_k^{**}$ as $N\to +\infty$. We also have a $\Gamma$-convergence result:

\begin{thm}\label{gammac}
As $N\to\infty$, the sequence of functionals $C_{N,k}$ defined in \eqref{mmk} $\Gamma$-converges with respect to the narrow topology of $\PP(X)$ to $P_k^{**}$.
\end{thm}

\begin{proof}
The $\Gamma$-limsup inequality obviously follows from \pref{upperbcc}. As for the $\Gamma$-liminf inequality, it follows directly from \pref{uniformcv} and the lower semicontinuity of $P_k^{**}$ with respect to the narrow topology.
\end{proof}

To conclude, we see that the asymptotic problem obtained by letting $N\to \infty$ in \pref{mmknr} amounts to computing the convex envelope of the polynomial  of degree $k$ interaction functional $P_k$. This might be a challenging computational task in general but we believe that the theory developed by Lasserre for polynomial optimization (see in particular \cite{Laslar}) might be useful in certain situations.

{\bf Acknowledgments:} G.C. gratefully acknowledges the hospitality of TU Munich, where this work was begun while he held a John Von Neumann visiting professorship.

\bibliographystyle{plain}

\bibliography{bibli}

\end{document}